\documentclass[a4paper,10pt,reqno]{amsart}
\usepackage[body={13cm,21cm}]{geometry}
\usepackage[initials]{amsrefs}

\usepackage{amssymb,amscd}
\usepackage[all]{xy}

\usepackage[mathscr]{eucal}
\usepackage{enumerate}

\numberwithin{equation}{section}
\theoremstyle{plain}
\newtheorem{thm}[equation]{Theorem}
\newtheorem{cor}[equation]{Corollary}
\newtheorem{lem}[equation]{Lemma}
\newtheorem{prop}[equation]{Proposition}

\newtheorem{definition}[equation]{Definition}
\newtheorem{exa}[equation]{Example}
\newtheorem{rem}[equation]{Remark}

\theoremstyle{definition}

\theoremstyle{remark}

\def\Q{{\mathbb Q}}
\def\Z{{\mathbb Z}}

\def\G{{\mathbb G}}

\def\A{{\mathbb A}}

\DeclareFontFamily{U}{wncy}{}
\DeclareFontShape{U}{wncy}{m}{n}{%
<5>wncyr5%
<6>wncyr6%
<7>wncyr7%
<8>wncyr8%
<9>wncyr9%
<10>wncyr10%
<11>wncyr10%
<12>wncyr6%
<14>wncyr7%
<17>wncyr8%
<20>wncyr10%
<25>wncyr10}{}
\DeclareMathAlphabet{\cyr}{U}{wncy}{m}{n}
\begin{document}

\title[Integral Points for Multi-norm Tori]
{Integral Points for Multi-norm Tori}

\author{Dasheng Wei$^1$}
\author{Fei Xu$^2$}

\address{$^1$ Academy of Mathematics and System Science,  CAS, Beijing
100190, P.R.China}

\email{dshwei@amss.ac.cn}

\address{$^2$ School of Mathematics, Capital Normal University,
Beijing 100048, P.R.China}

\email{xufei@math.ac.cn}

\date{\today}

\maketitle

 \bigskip

\section*{\it Abstract}

We construct a finite subgroup of Brauer-Manin obstruction for
detecting the existence of integral points on integral models of
principle homogeneous spaces of multi-norm tori. Several explicit
examples are provided.

\bigskip

{\it MSC classification} : 11D57, 11E72, 11G35, 11R37, 14F22, 14G25,
20G30


\bigskip

{\it Keywords} :  integral point, multi-norm torus, Galois
cohomology, Brauer-Manin obstruction.

\section*{Introduction} \label{sec.notation}

The integral points on homogeneous spaces of semi-simple and simply
connected linear algebraic groups of non-compact type were studied
by Borovoi and Rudnick in \cite{BR95} and by Colliot-Th\'el\`ene and
the second named author in \cite{CTX} by using the strong
approximation theorem and the Brauer-Manin obstruction. Recently,
Harari \cite{Ha08} showed that the Brauer-Manin obstruction accounts
for the nonexistence of integral points. Colliot-Th\'el\`ene noticed
that a finite subgroup of the Brauer group is enough to account for
the nonexistence of integral points by the compactness argument.
However this result is nonconstructive: one does not know which
finite subgroup to use and cannot use it to determine the existence
of integral points on the specific equations. In this paper, we give
some explicit construction for such finite subgroups for multi-norm
tori. The paper is also inspired by Colliot-Th\'el\`ene's suggestion
of studying Gauss' idea for determining integers represented by
positive definite binary quadratic forms, which is beautifully
explained by Cox in \cite{Co89}, from the point of view of
Brauer-Manin obstruction. The advantage of using Brauer-Manin
obstruction is to provide more perspective. For example, one can
determine the solvability of the negative Pell equations by using
the class field theory instead of the continued fractional method
(the quadratic Diophantine approximation).

The paper is organized as follows. In Section 1, we construct the
idele groups which are the so-called $\bf X$-admissible groups for
determining the integral points for some integral model $\bf X$. In
Section 2, we interpret the $\bf X$-admissible subgroup in terms of
finite Brauer-Manin obstruction and also explain that there is no
finite Brauer-Manin obstruction to detect all separated integral
models of finite type. In Section 3, we apply our construction to
study the classical binary quadratic Diophantine equations. Such
classical quadratic Diophantine equations have been studied for a
long time by various methods. This construction is the natural
extension of Gauss' method for determining integers represented by a
positive definite integral binary quadratic form and one can further
determine all integers represented by a given binary inhomogeneous
quadratic Diophantine equation. In Section 4, we provide several
examples of 1-dimensional non-split tori where the splitting fields
are imaginary quadratic fields. In Section 5, some examples of
1-dimensional non-split tori where the splitting fields are real
quadratic fields are studied. In Section 6, we explain how to apply
our construction to study the high dimensional multi-norm tori by
providing some more explicit examples.

Notation and terminology are standard if not explained. Let $F$ be a
number field, $\frak o_F$ be the ring of integers of $F$, $\Omega_F$
be the set of all primes in $F$ and $\infty_F$ be the set all
infinite primes in $F$. For simplicity, we write $\frak p<\infty_F$
for $\frak p\in \Omega_F\setminus \infty_{F}$. Let $F_\frak p$ be
the completion of $F$ at $\frak p$ and $\frak o_{F_\frak p}$ be the
local completion of $\frak o_F$ at $\frak p$ for each $\frak p\in
\Omega_F$. Write $\frak o_{F_\frak p}=F_\frak p$ for $\frak p\in
\infty_F$. For any positive integer, we use $(F^\times)^k$ (resp.
$({F_\frak p}^\times)^k$) to denote the subgroup of $k$-th powers in
$F^\times$ (resp. ${F_\frak p}^\times$). We also denote the adele
(resp. the idele) of $F$ by $\Bbb A_F$ (resp. $\Bbb I_F$) and
$$F_{\infty}=\prod_{\frak p\in \infty_F}F_\frak p .$$

Let $E_1, \cdots, E_n$ be the finite extensions of $F$ and
$$ \varphi: \ {R_{E_1/F}(\G_m)\times \cdots\times
R_{E_n/F}(\G_m)\longrightarrow \G_m}$$ be the homomorphism of
algebraic groups which represents
$$ (E_1 \otimes _F A)^\times  \times \cdots \times  (E_n \otimes _F
 A)^\times \rightarrow A^\times ; \ \ (x_1, \dots, x_n) \mapsto \prod_{i=1}^n N_{E_i/F}(x_i)  $$
for any $F$-algebra $A$.  Define $G:=ker \varphi$ which is called a
multi-norm torus over $F$. Let $\hat{G}=Hom_{\bar{F}}(G, \Bbb G_m)$
be the characters of $G$ where $\bar{F}$ is an algebraic closure of
$F$ and $\bf X$ be a separated $\frak o_F$-scheme of finite type
whose generic fiber $X_F$ is a principal homogeneous space of $G$.
The obvious necessary conditions for ${\bf X}(\frak o_F)\neq
\emptyset$ are
\begin{equation}\label{loc} \prod_{\frak p\in \Omega_F} {\bf
X}(\frak o_{F_\frak p})\neq \emptyset \ \ \ \text{and} \ \ \
X_F(F)\neq \emptyset
\end{equation} which is assumed throughout this paper. It should pointed out that the determination of
$X_F(F)\neq \emptyset$ or not is well understood (see \cite{BCS}).

The Brauer group $Br(X_F)$ of $X_F$ is defined as
$$Br(X_F)=H^2_{et}(X_F, \Bbb G_m) \ \ \ \text{and} \ \ \
Br_1(X_F)=ker[Br(X_F)\rightarrow Br(\bar{X})]$$ where
$\bar{X}=X_F\times_F \bar{F}$. Since the image of $Br(F)$ induced by
the structure morphism lies in $Br_1(X_F)$, one defines
$$Br_a(X_F)=coker[Br(F)\rightarrow Br_1(X)]. $$
For any subgroup $\frak s$ of $Br_a (X_F)$, one can define the
integral Brauer-Manin set with respect to $\frak s$ as (see
\cite{CTX})
$$(\prod_{\frak p\in \Omega_F}{\bf X}(\frak o_{F_\frak p}))^{\frak
s}=\{ (x_\frak p)\in \prod_{\frak p\in \Omega_F}{\bf X}(\frak
o_{F_\frak p}): \ \ \sum_{\frak p\in \Omega_F} inv_\frak p(s(x_\frak
p))=0, \ \ \ \forall s\in \frak s\}.$$

\bigskip

\section{ Construction of $\bf X$-admissible Groups  } \label{sec.adm}

Since $\bf X$ is separated over $\frak o_{F}$, one can view ${\bf
X}(\frak o_{F_\frak p})$ as an open subset of $X_F(F_\frak p)$ by
the natural map for any $\frak p<\infty_F$ and $G(F_\frak p)$ acts
on $X_F(F_\frak p)$ continuously.

\begin{definition}\label{stm} Let $$Stab({\bf X}(\frak o_{F_\frak p}))=\{ g\in
G(F_\frak p):  \ \ g{\bf X}(\frak o_{F_\frak p})={\bf X}(\frak
o_{F_\frak p})\}$$ for $\frak p<\infty_F$ and $$Stab({\bf X}(\frak
o_{F_\frak p}))=G(F_\frak p) \ \ \ \ \ \text{for \ $\frak p\in
\infty_F$}.
$$ Define

$$Stab_{\Bbb A}({\bf
X})= G(\Bbb A_F)\cap  [\prod_{\frak p\in \Omega_F}Stab({\bf X}(\frak
o_{F_\frak p}))].$$
\end{definition}

One has the following basic property of $Stab_{\Bbb A}({\bf X})$.

\begin{lem}\label{stab} $Stab_{\Bbb A}({\bf
X})$ is an open subgroup of $G(\Bbb A_F)$.
\end{lem}
\begin{proof} Since the stabilizer of an open subset is open by the continuality,
 one obtains that $Stab({\bf X}(\frak o_{F_\frak p}))$ is an open
 subgroup of $G(F_\frak p)$ for all $\frak p\in \Omega_F$.

 Let $S$ be a finite subset of $\Omega_F$ containing $\infty_F$ such
 that $G$ has a model $\bf G$ over $\frak o_S$ and the morphism $G\times_F X_F \rightarrow X_F$ can
 be extended to ${\bf G}\times_{\frak o_S}{\bf X}_S \rightarrow {\bf X}_S $ where $\frak o_S$ is the $S$-integers of $F$
 and ${\bf X}_S={\bf X}\times_{\frak o_F} \frak
 o_S$. This implies
 that $$Stab({\bf X}(\frak o_{F_\frak p}))\supseteq {\bf G}(\frak o_{F_\frak
 p})  \ \ \ \text{for all $\frak p\not\in S$}. $$ Therefore one
 concludes that $Stab_{\Bbb A}({\bf
X})$ is an open subgroup of $G(\Bbb A_F)$.
\end{proof}

Since $G$ is the subgroup of $R_{E_1/F}(\G_m)\times \cdots\times
R_{E_n/F}(\G_m)$ by the definition of $G$, one has the injective map
$$\lambda_{(E_1,\dots,E_n)}: \ \  G(A) \longrightarrow (E_1 \otimes _F A)^\times  \times \cdots \times  (E_n \otimes _F
 A)^\times  $$ for any $F$-algebra $A$.

\begin{definition}\label{admmodel} An open subgroup $\Xi$ of $\prod_{i=1}^n \Bbb
I_{E_i}$ is called $\bf X$-admissible if
$$\lambda_{(E_1,\dots,E_n)}[Stab_\A({\bf X})]\subseteq \Xi$$ and the map induced by $\lambda_{(E_1,\dots,E_n)}$
$$ \lambda_{(E_1,\dots,E_n)}: \ \ \ G(\Bbb A_F)/G(F)Stab_\A({\bf X}) \longrightarrow
\prod_{i=1}^n \Bbb I_{E_i}/(\prod_{i=1}^n E_i^\times)\cdot \Xi $$ is
injective.

\end{definition}

Since $X_F$ is a trivial $G$-torsor over $F$, one can fix an
isomorphism $X_F\cong G$ induced by a rational point $P$. Combining
with $\lambda_{(E_1,\dots,E_n)}$, one can define the injective map
$$f_{(E_1,\dots, E_n)}: \ \ X_{F}(A)
\cong G(A) \xrightarrow{\lambda_{(E_1,\cdots,E_n)}}  \prod_{i=1}^n
(E_i\otimes_F A)^\times
$$ for any $F$-algebra $A$.

Since $\bf X$ is separated over $\frak o_F$, one can regard
$\prod_{\frak p\in \Omega_F} \bold X(\frak o_{F_\frak p})$ as the
subset of $X_F(\Bbb A_F)$ by the natural morphism. Then
$$f_{(E_1,\dots, E_n)}[\prod_{\frak p\in \Omega_F} \bold X(\frak o_{F_\frak p})]\subseteq \prod_{i=1}^n \Bbb I_{E_i}. $$

The following proposition explains the reason for introducing the
concept of $\bf X$-admissible subgroups.

\begin{prop} \label{mainmodel} Let $\Xi$ be an $\bf X$-admissible subgroup of
$\prod_{i=1}^n \Bbb I_{E_i}$. Then
$$\bold X(\frak o_F) \neq \emptyset \ \ \ \text{if and only if} \ \
\ f_{(E_1,\dots, E_n)}[\prod_{\frak p\in \Omega_F} \bold X(\frak
o_{F_\frak p})]\cap [(\prod_{i=1}^n E_i^\times )\cdot \Xi] \neq
\emptyset.$$
\end{prop}
\begin{proof} Since $\bold X$ is separated over $\frak o_F$, one has
$\bold X(\frak o_F)\subseteq \bold X(\frak o_{F_\frak p})$ for all
$\frak p\in \Omega_F$ and
$$\bold X(\frak o_F) \subseteq \prod_{\frak p\in \Omega_F} \bold X(\frak o_{F_\frak
p})$$ by the diagonal map. If $\bold X(\frak o_F)\neq \emptyset$,
then
$$f_{(E_1,\dots,E_n)}[\prod_{\frak p\in \Omega_F} \bold X(\frak o_{F_\frak p})]\cap
[(\prod_{i=1}^n E_i^\times )\cdot \Xi]\supseteq f_{(E_1,\dots,
E_n)}[\bold X(\frak o_F)]\cap (\prod_{i=1}^n E_i^\times ) \neq
\emptyset$$ and the necessity follows.

Conversely, assume that  $$y_A \in \prod_{\frak p\in \Omega_F} \bold
X(\frak o_{F_{\frak p}}) \ \ \text{ such that } \ \
f_{(E_1,\dots,E_n)}(y_A)\in (\prod_{i=1}^n E_i^\times )\cdot \Xi .$$
By Definition \ref{admmodel}, there are $\varrho \in G(F)$ and
$\sigma_A\in Stab_\A({\bf X})$ such that $y_A=\varrho \sigma_A (P)$
by injectivity of $f_{(E_1,\dots,E_n)}$. This implies that $$\varrho
(P)=\sigma_A^{-1} (y_A)\in \prod_{\frak p\in \Omega_F} \bold X(\frak
o_{F_\frak p}).
$$ Therefore $\varrho (P) \in \bold X(\frak o_F)\neq \emptyset$  and
the proof is complete.
\end{proof}

If $\Xi$ is an $\bf X$-admissible subgroup of $\prod_{i=1}^n \Bbb
I_{E_i}$, there is an open subgroup $\Xi_i$ of $\Bbb I_{E_i}$ for
each $1\leq i\leq n$ such that $\prod_{i=1}^n \Xi_i \subseteq \Xi$.
By the class field theory, there is a finite abelian extension
$K_{\Xi_i}/E_i$ such that the Artin map
\begin{equation}\label{artinmulti} \psi_{K_{\Xi_i}/E_i}: \ \ \ \Bbb
I_{E_i}/E_i^\times \Xi_i \cong Gal(K_{\Xi_i}/E_i)
\end{equation} gives the isomorphism for $1\leq i\leq n$. Projecting
the image of $f_{(E_1,\dots,E_n)}$ to $\Bbb I_{E_i}$, one can define
$$f_{E_i}: \ \ \prod_{\frak p\in \Omega_F} \bold X(\frak o_{F_\frak p}) \longrightarrow \prod_{j=1}^n \Bbb I_{E_j}
\longrightarrow \Bbb I_{E_i} $$ for $1\leq i\leq n$.

\begin{cor}\label{efmulti} With the notation as above,
$\bold X(\frak o_F)\neq \emptyset$ if and only if there is
$$x_A\in\prod_{\frak p\in \Omega_F} \bold X(\frak o_{F_\frak p}) \ \ \ \text{such that} \ \ \
\psi_{K_{\Xi_i}/E_i}(f_{E_i}(x_A))=1 $$ in $Gal(K_{\Xi_i}/E_i)$ for
all $1\leq i\leq n$.
\end{cor}
\begin{proof} Since $\psi_{K_{\Xi_i}/E_i}(f_{E_i}(x_A))=1$ for $1\leq i\leq n$, one
has $$f_{(E_1,\dots,E_n)}(x_A)\in (\prod_{i=1}^n E_i^\times) \cdot
(\prod_{i=1}^n \Xi_i)\subseteq (\prod_{i=1}^n E_i^\times)\cdot \Xi
.$$ Then the result follows from the same argument as those in Prop.
\ref{mainmodel}.
 \end{proof}

The main purpose in this section is to construct such $\bf
X$-admissible groups. First we need to the following lemma.

\begin{lem}\label{openmodel} If $C_\frak p$ is an open subgroup of $G(F_\frak p)$
and $k$ is a positive integer, then
$$\lambda_{(E_1,\dots,E_n)}(C_\frak p) \cdot \{(a, \dots, a): \ a \in (F_\frak
p^\times)^{k}\}
$$ is an open subgroup of $\prod_{i=1}^n {E_i}_\frak p^\times$ for
$\frak p$-adic topology with $\frak p<\infty_F$.
\end{lem}
\begin{proof} Since $C_\frak p$ is an open subgroup of $G(F_\frak
p)$, there are positive integers $a_1, \dots, a_n$ such that
$$[\prod_{i=1}^n(1+\frak p^{a_i} \frak o_{{E_i}_{\frak p}})]\cap \lambda_{(E_1,\dots,E_n)}(G(F_\frak p))
\subseteq \lambda_{(E_1,\dots,E_n)}(C_\frak p).
$$ By Hensel's lemma, there is a sufficiently large positive integer
$b_i>a_i$ such that $$(1+\frak p^{b_i} \frak o_{{E_i}_{\frak
p}})\subseteq (1+\frak p^{a_i} \frak o_{{E_i}_{\frak p}})^{lk}$$
where $l=\sum_{i=1}^n l_i$ and $l_i=[E_i:F]$ for all $1\leq i\leq
n$.

If  $$(x_i)_{i=1}^n \in \prod_{i=1}^n (1+\frak p^{b_i} \frak
o_{{E_i}_{\frak p}}) , $$ then
$$(x_i)_{i=1}^n=(y_i^{lk})_{i=1}^n \ \ \ \text{with  \ \ \
$(y_i)_{i=1}^n \in \prod_{i=1}^n(1+\frak p^{a_i} \frak
o_{{E_i}_{\frak p}})$} .$$ Since
$$y_i^{lk}=[y_i^{lk} (\prod_{i=1}^nN_{E_i/F}(y_i))^{-k}]
\cdot [\prod_{i=1}^nN_{E_i/F}(y_i)]^k$$ and
$$([y_i^{lk}(\prod_{i=1}^nN_{E_i/F}(y_i))^{-k}])_{i=1}^n \in
\lambda_{(E_1,\dots,E_n)}(C_\frak p),$$ one concludes that
$$\prod_{i=1}^n (1+\frak p^{b_i} \frak o_{{E_i}_{\frak p}})\subseteq
\lambda_{(E_1,\dots,E_n)}(C_\frak p) \cdot \{(a, \dots, a): \ a \in
(F_\frak p^\times)^{k}\}
$$ and the proof is complete.
\end{proof}

In order to prove the existence of $\bf X$-admissible groups, we
further need the following lemma.

\begin{lem} \label{power} Let $k$ be a positive integer. Then there is
a finite subset $T_0$ of $\Omega_F\setminus \infty_F$ such that
$$F^\times \cap [\prod_{\frak p\in T} (F_\frak p^\times)^{2kh_F} \times \prod_{\frak p\not\in T}\frak o_{F_\frak
p}^\times]
 \subseteq (F^\times)^k $$ for any $T\supseteq T_0$, where $h_F$ is the class number of $F$.
\end{lem}
\begin{proof} Since $\frak o_F^\times/(\frak o_F^\times)^{2k}$ is finite by the Dirichlet unit theorem,
there are only finitely many cosets $\alpha_i (\frak
o_F^\times)^{2k}$ of $\frak o_F^\times/(\frak o_F^\times)^{2k}$ such
that $\alpha_i\not\in (F_{\frak p_i}^\times)^{2k}$ for some prime
$\frak p_i <\infty_F$. Fix one such prime $\frak p_i$ for each
$\alpha_i (\frak o_F^\times)^{2k}$ and define $T_0$ to be the set
consisting of all such $\frak p_i$. It is possible that $T_0$ is
empty.

 Let
$$ \alpha \in F^\times \cap [\prod_{\frak p\in T} (F_\frak
p^\times)^{2kh_F} \times \prod_{\frak p\not\in T}\frak o_{F_\frak
p}^\times] $$ for any finite set $T\supseteq T_0$. Since $\frak
p^{h_F}$ is a principal ideal, there is $\varpi_\frak p\in F$ such
that $\frak p^{h_F}=\varpi_\frak p\frak o_F$ for each finite $\frak
p\in T$. There is $t_\frak p\in \Bbb Z$ for each finite $\frak p\in
T$ such that
$$\beta=\alpha ( \prod_{\frak p\in T\setminus \infty_F} \varpi_\frak p^{t_\frak p})^{2k} \in
\frak o_F^\times .$$ Then $\beta(\frak o_F^\times)^{2k}$ is not one
of the above mentioned cosets $\alpha_i (\frak o_F^\times)^{2k}$.
This implies that
$$ \alpha \in F^\times \cap [ \prod_{\frak p\in \infty_F}F_\frak
p^\times \times \prod_{\frak p\not \in \infty_F} (F_\frak
p^\times)^{2k} ] .
$$ Therefore
$$F^\times \cap [\prod_{\frak p\in T} (F_\frak
p^\times)^{2kh_F} \times \prod_{\frak p\not\in T}\frak o_{F_\frak
p}^\times] \subseteq F^\times \cap [ \prod_{\frak p\in
\infty_F}F_\frak p^\times \times \prod_{\frak p\not \in \infty_F}
(F_\frak p^\times)^{2k} ] .$$

Let $$x\in F^\times \cap [ \prod_{\frak p\in \infty_F}F_\frak
p^\times \times \prod_{\frak p\not \in \infty_F} (F_\frak
p^\times)^{2k} ] .$$ Applying (9.1.3) Theorem in \cite{NSW} for
$\mu_{2k}$, one obtains that $x^2\in (F^\times)^{2k}$. There is
$y\in F^\times$ such that $x=y^k$ or $x=-y^k$.

Suppose $x=-y^k$. Let $k=2^s k_1$ with $2\nmid k_1$ and
$\zeta_{2^{s+1}}$ be a primitive $2^{s+1}$-th roots of unity. Then
$\zeta_{2^{s+1}}\in F_\frak p^\times$ for all $\frak p<\infty_F$. By
the Chebotarev density theorem, one concludes that
$\zeta_{2^{s+1}}\in F$. Therefore $-1=(\zeta_{2^{s+1}})^k \in
(F^\times)^k$ and the proof is complete.
\end{proof}

The proof of Lemma \ref{power} also provide the explicit method to
determine the set $T_0$. For example, when $F=\Bbb Q$, one has
$T_0=\{2\}$ or $\{p\}$ with $p\equiv 3 \mod 4$. The following
corollary has its own independent interest.

\begin{cor}
For any open subgroup $\Xi$ of $\Bbb I_F$ and a positive integer
$k$, there is an open subgroup $\Xi_k$ of $\Bbb I_F$ such that
$(F^\times \cap \Xi_k)\subseteq (F^\times \cap \Xi)^k $.
\end{cor}

\begin{proof}
Let $\Xi_\frak p$ be an open subgroup of $F_\frak p$ for each $\frak
p\in \Omega_F$ such that $\prod_{\frak p\in \Omega_F}\Xi_{\frak p}
\subseteq \Xi $ and $\Xi_\frak p=\frak o_{F_\frak p}^\times$ for
almost all $\frak p \in \Omega_F$. Let $S_0$ be a finite subset of
$\Omega_F$ such that $\Xi_\frak p=\frak o_{F_\frak p}^\times$ for
all $\frak p\not\in S_0$. For each $\frak p\in S_0$ and for a prime
divisor $q$ of $k$, one defines $t_\frak p(q)$ to be the maximal
positive integer $t$ such that the primitive $q^t$-th roots of unity
are in $F_\frak p$. Let
$$t_\frak p: = \max \{ t_\frak p(q): \ \text{prime divisor $q$ of $k$} \} \ \ \ \text{and} \ \ \
r: =\prod_{\frak p\in S_0}(t_\frak p+1) .$$ Applying Lemma
\ref{power} for $k^{r+1}$, one obtains a finite subset $T_0$ of
$\Omega_F$. Let $S=S_0\cup T_0$ and defines $$\Xi_k =\prod_{\frak
p\in S}\Xi_\frak p^{2h_Fk^{r+1}} \times \prod_{\frak p\not\in
S}\Xi_\frak p  .$$

For any $x\in (F^\times \cap \Xi_k)$, there is $y\in F^\times$ such
that $x=y^{k^{r+1}}$ by Lemma \ref{power}. Moreover, there is
$\xi_\frak p\in \Xi_\frak p$ such that $$x=y^{k^{r+1}}=\xi_\frak
p^{k^{r+1}}$$ for all $\frak p\in S$.

If $\frak p\in S_0$, one has that $$y^{k^r}=\xi_\frak p^{k^r} \in
\Xi_\frak p
$$ by the maximality of $t_\frak p$.

If $\frak p \not \in S_0$, one has $$y^{k^r}\in \frak o_{F_\frak
p}^\times =\Xi_\frak p  . $$

Therefore $y^{k^r}\in (F^\times \cap \Xi)$ and the proof is
complete.
\end{proof}

Now we can show the existence of the admissible subgroups for $\bf
X$.

\begin{thm}\label{multiexsitence} The $\bf X$-admissible subgroups of $\prod_{i=1}^n\Bbb I_{E_i}$
always exist. \end{thm}
\begin{proof}
Since $X_F$ is a trivial torsor  over $F$ under $G$, there is a
finite subset $S_0 \subset (\Omega_F\setminus \infty_F)$ such that

1) the map $\varphi$ can be extended to
$$  \varPhi : \ {R_{\frak o_{E_1,S_0}/\frak o_{F,S_0}}(\G_m)\times \cdots\times
R_{\frak o_{E_n,S_0}/\frak o_{F,S_0}}(\G_m)\longrightarrow \G_m}$$
over $\frak o_{F,S_0}$, where $\frak o_{F,S_0}$, and $\frak
o_{E_1,S_0}, \cdots \frak o_{E_n, S_0}$ are the $S_0$-integers of
$F$ and $E_1, \cdots, E_n$ respectively. Moreover, $ker(\varPhi)=\bf
G$ is a model of $G$ over $\frak o_{F, S_0}$.

2)  ${\bf X}_{S_0}$ is a trivial torsor over $\frak o_{S_0}$ under
$\bf G$, where ${\bf X}_{S_0}={\bf X}\times_{\frak o_F} \frak
 o_{F, S_0}$.

This implies that
$$\lambda_{(E_1,\dots,E_n)} [Stab({\bf X}(\frak o_{F_\frak p}))]=
ker (\prod_{i=1}^n N_{E_i/F}: \prod_{i=1}^n\frak o_{{E_i}_\frak
p}^\times {\longrightarrow} \ \frak o_{F_\frak p}^\times)$$ for
every $\frak p\not \in S_0$.

Let $l=\sum_{i=1}^n l_i$ with $l_i=[E_i:F]$ for all $1\leq i\leq n$.
For each $\frak p\in S_0$ and prime divisor $q$ of $l$, one defines
$t_\frak p(q)$ to be the maximal positive integer $t$ such that the
primitive $q^t$-th roots of unity are in $F_\frak p$. Let
$$t_\frak p: = \max \{ t_\frak p(q): \ \text{$q$ divides $l$} \} \ \ \ \text{and} \ \ \
r: =\prod_{\frak p\in S_0}(t_\frak p+1) .$$

Let $T_0$ be a finite subset of $\Omega_F$ outside $\infty_F$  such
that Lemma \ref{power} holds for $k=l^{r+1}$. Put $S=S_0\cup T_0$
and
$$\Xi= [\prod_{\frak p\in S}\lambda_{(E_1,\dots,E_n)}(Stab({\bf X}(\frak o_{F_\frak p})))\cdot H_\frak
p] \times (\prod_{i=1}^n \prod_{\frak p\not \in S} \frak
o_{{E_i}_\frak p}^\times)
$$ where $$H_\frak p=\{(a,\dots, a): a\in (F_\frak p^\times)^{2l^rh_F}
\}$$ and $h_F$ is the class number of $F$. By Lemma \ref{openmodel},
one has that $\Xi$ is an open subgroup of $\prod_{i=1}^n \Bbb
I_{E_i}$ and
$$\lambda_{(E_1,\dots,E_n)}[Stab_\A({\bf X})]\subseteq \Xi .$$

Suppose $\sigma \in G(\A_F)$ with $\lambda_{(E_1,\dots,E_n)}(\sigma)
=ai$ where
$$a=(a_i)_{i=1}^n \in \prod_{i=1}^n E_i^\times \ \ \ \text{and} \ \ \ i\in \Xi . $$ Then
$$ \prod_{i=1}^n N_{E_i/F}(a_i)\in [\prod_{\frak p\in S}(F_\frak p^\times)^{2l^{r+1}h_F} \times
\prod_{\frak p\not\in S}\frak o_{F_\frak p}^\times ].$$  By Lemma
\ref{power}, there is $u\in F^\times$ such that $\prod_{i=1}^n
N_{E_i/F}(a_i)=u^{l^{r+1}}$. This implies that
$$(a_1 u^{-l^r}, \dots , a_n u^{-l^r}) \in \lambda_{(E_1,\dots,E_n)}(G(F)) . $$ Write $i=( i_\frak p )_{\frak
p\in \Omega_F}$ and $i_\frak p=s_\frak p \cdot (n_\frak p, \dots,
n_{\frak p})$ with
$$s_\frak p\in \lambda_{(E_1,\dots,E_n)}(Stab({\bf X}(\frak o_{F_\frak p})) \ \ \ \text{and} \ \ \
n_\frak p\in (F_\frak p^\times)^{2l^rh_F}$$ for $\frak p\in S$.
Therefore $u^{l^{r+1}}n_\frak p^l=1$ for $\frak p\in S$.

If $\frak p\in S_0$, one has that $u^{l^r}n_\frak p=1$ by the
maximality of $t_\frak p$.

If $\frak p \in S\setminus S_0$, one has $$u^{l^r}(n_\frak p, \dots,
n_\frak p)  \in ker (\prod_{i=1}^n N_{E_i/F}: \prod_{i=1}^n\frak
o_{{E_i}_\frak p}^\times {\longrightarrow} \ \frak o_{F_\frak
p}^\times) .$$

Therefore one concludes that
$$u^{l^r}i\in \lambda_{(E_1,\dots,E_n)}[Stab_\A({\bf X})] \ \ \ \text{and} \ \ \ \sigma
\in G(F)Stab_\A({\bf X}) $$  and the proof is complete.
\end{proof}

\bigskip

\section{Finite Brauer-Manin Obstruction}

In this section we interpret the $\bf X$-admissible subgroup in
terms of Brauer-Manin obstruction. We keep the same notation as
before. By the definition of $G$, one has the surjective map
$$\widehat{R_{E_1/F}(\Bbb G_m)}\times \cdots \times \widehat{R_{E_n/F}(\Bbb
G_m)}\longrightarrow \hat{G} $$ which induces
\begin{equation}  \phi : \ \ \
\label{prod} \prod_{i=1}^n H^2(F, \widehat{R_{E_i/F}(\Bbb
G_m)})=\prod_{i=1}^n H^2(E_i, \Bbb Z)\longrightarrow H^2(F,\hat{G})
\end{equation}  by Shapiro's lemma (see (1.6.3) Proposition in
\cite{NSW}).

Let $\Xi$ be an $\bf X$-admissible subgroup of $\prod_{i=1}^n \Bbb
I_{E_i}$. There is an open subgroup $\Xi_i$ of $\Bbb I_{E_i}$ for
each $1\leq i\leq n$ such that $\prod_{i=1}^n \Xi_i \subseteq \Xi$.
By the class field theory, there is a finite abelian extension
$K_{\Xi_i}/E_i$ satisfying (\ref{artinmulti}) for $1\leq i\leq n$.

By the short exact sequence
$$0\longrightarrow \Bbb Z\longrightarrow \Bbb Q \longrightarrow \Bbb
Q/\Bbb Z \longrightarrow 0 $$ with the trivial action of
$Gal(K_{\Xi_i}/E_i)$, one obtains an isomorphism of finite groups
\begin{equation} \label{connection} \delta_i: \ \
Hom(Gal(K_{\Xi_i}/E_i), \Bbb Q/\Bbb Z)   \cong
H^2(Gal(K_{\Xi_i}/E_i), \Bbb Z).
\end{equation}  Let $b_i(\Xi_i)$ be the image of the
inflation \begin{equation} \label{inf} H^2(Gal(K_{\Xi_i}/E_i), \Bbb
Z)\longrightarrow H^2(E_i, \Bbb Z)  \end{equation} for $1\leq i\leq
n$. Since $X_F\cong G$ by the fixed rational point $P$, one has $$
\rho: \ \ H^2(F, \hat{G}) \cong Br_1(G) \cong Br_1(X_F)
 $$
by Theorem 1 of \S 4.3 in \cite{Vo98}. Let $b(\Xi)$ be the subgroup
of $Br_1(X_F)$ generated by the image of $\{b_i(\Xi_i)\}_{i=1}^n $
under the map $\rho \circ \phi$. It is clear that $b(\Xi)$ is
finite.

\begin{prop}\label{brauer} With the above notation,
$$\bold X(\frak o_F) \neq \emptyset \ \ \ \text{ if and only if} \ \
\ \ [\prod_{\frak p\in \Omega_F} \bold X(\frak o_{F_\frak
p})]^{b(\Xi)}\neq \emptyset.$$
\end{prop}
\begin{proof} One only needs to show the sufficiency.
For any $$\chi \in Hom(Gal(K_{\Xi_i}/E_i), \Bbb Q/\Bbb Z) \ \
\text{and} \ \ (x_\frak p)_{\frak p\in \Omega_F}\in [\prod_{\frak
p\in \Omega_F} \bold X(\frak o_{F_\frak p})]^{b(\Xi)} , $$ there is
$(\sigma_\frak p)_{\frak p\in \Omega_F}\in G(\Bbb A_F)$ such that
$$f_{(E_1,\dots,E_n)}((x_\frak p)_{\frak p\in \Omega_F})=\lambda_{(E_1,\dots, E_n)}((\sigma_\frak p)_{\frak p\in \Omega_F}).$$
Since
$$ H^1(K_{\Xi_i}, \Bbb Z)\cong Hom_{cont}(Gal(\bar E/K_\Xi), \Bbb
Z)=0 , $$ the inflation map in (\ref{inf}) is injective. One can
view $\delta_i(\chi)\in b_i(\Xi_i)$ and obtains
$$\begin{aligned}
 & 0=\sum_{\frak p\in \Omega_F}inv_\frak p(\rho \circ
\phi(\delta_i(\chi)) (x_\frak p))= \sum_{\frak p\in
\Omega_F}inv_\frak p(\phi(\delta_i(\chi)) (\sigma_\frak p)) \\
= & \sum_{\frak p\in \Omega_F}inv_\frak p (\delta_i(\chi)\cup
(\lambda_{(E_1,\dots,E_n)}(\sigma_\frak
p)))=\chi(\psi_{K_{\Xi_i}/E_i}
(\lambda_{(E_1,\dots,E_n)}((\sigma_\frak p)_{\frak p\in \Omega_F})))
\end{aligned}
$$
by the functoriality of the Brauer-Manin pairing (see (5.3) in
\cite{Sko}), the compatibility of Bauer-Manin pairing with the cup
product (see  Section 6 in \cite{HS08}) and (8.1.11) Proposition in
\cite{NSW}. This implies
$$\psi_{K_{\Xi_i}/E_i}(f_{E_i}((x_\frak p)_{\frak p\in \Omega_F})=0 $$
for all $1\leq i\leq n$. The result follows from Corollary
\ref{efmulti}.
\end{proof}

It should be pointed out that the finite subgroup $b(\Xi)$ of
$Br_1(X_F)$ depends on the integral model $\bf X$. In general, there
does not exist a universal finite subgroup $\frak s$ of $
Br_1(X_F)$, which can be used to test the existence of the integral
points for all integral model $\bf X$ of $X_F$.

\begin{exa} Let $F=\Q$ and $E=\Q(\sqrt{-1})$ and
$$G=X_\Q=R^1_{E/F}(\G_m)=Spec(\Q[x,y]/(x^2+y^2-1)). $$
For any finite subgroup $\frak s$ of $Br_1(X_\Q)$, there exists a
separated $\Bbb Z$-scheme $\bf X$ of finite type with the generic
fiber $X_\Q$ such that $$(\prod_{p\leq \infty}{\bf X}(\Bbb
Z_p))^{\frak s} \neq \emptyset \ \ \ \ \text{but} \ \ \ \ {\bf
X}(\Bbb Z)=\emptyset .$$
\end{exa}

\begin{proof} Since $(x+\sqrt{-1}y)(x-\sqrt{-1}y)=x^2+y^2$, one has
$$ \bar \Q[X_\Q]^\times =\bar \Q^\times \times f_E^\Bbb Z
\ \ \ \text{with}  \ \ \ f_E=x+y\sqrt{-1}. $$ Define
$$Ind_{Gal(\bar \Q/E)}^{Gal(\bar \Q/\Q)} (f_E^\Z) \longrightarrow
f_E^\Z, \ \ \ \xi\mapsto \sum_{\sigma\in Gal(E/\Q)}\sigma
\xi(\sigma^{-1}) .$$  One has the exact sequence of $Gal(\bar
\Q/\Q)$-module
$$1\longrightarrow \Z \longrightarrow Ind_{Gal(\bar \Q/E)}^{Gal(\bar
\Q/\Q)} (f_E^\Z) \longrightarrow f_E^\Z \longrightarrow 1 .$$ The
Galois cohomology gives
$$H^2(\Q,Ind_{Gal(\bar \Q/E)}^{Gal(\bar
\Q/\Q)} (f_E^\Z))\longrightarrow H^2(\Q,f_E^\Z)\longrightarrow
H^3(\Q,\Z).$$ By Corollary 4.17 in \cite{Milne86}, one has
$H^3(\Q,\Z)=0$. By Shapiro's lemma (see (1.6.4) Proposition
\cite{NSW}), the corestriction map  \begin{equation} \label{cor}
cor_{E/\Q}: \ \ \ H^2(E, f_E^\Z) \longrightarrow H^2(\Q,
f_E^\Z)\end{equation} is surjective. Since $Gal(\bar \Q/E)$ acts on
$f_E^\Z$ trivially, one gets
\begin{equation} \label{iso} \delta: \ \ \ Hom(Gal(\bar \Q/E), \Q/\Z)
\cong H^2(E, f_E^\Z)
 \end{equation} by using
the exact sequence with the trivial action $$0\longrightarrow \Z
\longrightarrow \Q \longrightarrow \Q/\Z \longrightarrow 0. $$

By Theorem 1 of \S 4.3 in \cite{Vo98}, there is a finite subgroup
$$S\subset Hom(Gal(\bar \Q/E), \Q/\Z) \ \ \ \text{ such that} \ \ \
cor_{E/\Q}\circ \delta (S)= \frak s . $$ By the class field theory,
there are a finite abelian extension $K_S/E$ such that
$$S=\{ \chi \in Hom(Gal(\bar \Q/E), \Q/\Z): \ \ \ \chi|_{Gal(\bar
\Q/K_S)}=1 \}$$ and an open subgroup $U$ of $\Bbb I_E$ with
$$  \psi_S: \ \ \ \Bbb I_E/E^\times U \cong Gal(K_S/E) $$ via the Artin map. Let
$l$ be an odd prime such that the invertible elements of the
$l$-adic completion of $\frak o_E$ are contained in $U$. Define $\bf
X$ by the equations
$$\begin{cases}x^2+y^2=1\\
x-l+l^2z=0.
\end{cases}$$
By Hensel's lemma, there is $y_0\in \Z_l$ such that
$(1-l)^2+2y_0+l^2 y_0^2=0$. Choose the local solution $(s_p)_p \in
\prod_{p\leq \infty} {\bf X}(\Z_p)  $ by
$$s_p=(x_p,y_p,z_p)=\begin{cases} (1,0,l^{-2}(l-1)) \ \ \ & \text{if $p\neq l$}  \\
(l(1-l), 1+l^2y_0, 1) \ \ \ & \text{if $p=l$.} \end{cases} $$ Since
$f_E(s_p)\in \prod_{\frak p |p}E_\frak p$ and more precisely
$$f_E((x_p,y_p,z_p))=\begin{cases} (x_p+y_p\sqrt{-1}, x_p-y_p\sqrt{-1}) \ \ \ & \text{if $p$
splits in $E/\Q$} \\
x_p+y_p\sqrt{-1} \ \ \ & \text{otherwise} \end{cases} $$ for all
primes $p$ including $\infty$, one has $f_E((s_p)_p)\in U$. For any
$\beta\in \frak s$, there is $\chi \in S$ such that $\beta
=cor_{E/\Q}\circ \delta (\chi)$.  By (8.1.4) Proposition and (7.1.4)
Corollary in \cite{NSW}, one has
$$\sum_{p\leq \infty} inv_p(\beta (s_p))=\sum_{\frak p\in \Omega_E}inv_\frak p(\delta(\chi)(s_p)).$$
By (8.1.11) Proposition in \cite{NSW}, one concludes that
$$\sum_{\frak p\in \Omega_E}inv_\frak p(\delta(\chi)(s_p))=\chi(\psi_S(f_{E}((s_p)_{p}))=0. $$
This implies that $$(s_p)_{p\leq \infty} \in (\prod_{p\leq \infty}
{\bf X}(\Z_p))^\frak s .$$ However, it is clear that ${\bf
X}(\Z)=\emptyset$.
\end{proof}
\bigskip

\section{Binary quadratic Diophantine equations} \label{quadratic}

In this section, we apply the results in the previous sections to
$\frak o_F$-scheme $\bold X$ defined by the following irreducible
polynomial
\begin{equation} \label{equ}
ax^2+bxy+cy^2+ex+fy+g=0 \end{equation}  with $a,b,c,e,f,g\in \frak
o_F$ and $-d=b^2-4ac \in F^\times$.

Let $V=Fv_1+Fv_2$ be a 2-dimensional quadratic space over $F$ with a
basis $\{v_1, v_2\}$ and the associated symmetric bilinear map $B:
V\times V\rightarrow F$ satisfying
$$  B(v_1,v_1)=a,  \ B(v_2,v_2)=c \ \
\text{and} \ \ B(v_1,v_2)=\frac{1}{2} b . $$ Let the group $G$
represent $SO(V\otimes_F A)$ for any $F$-algebra $A$. Then $X_F$ is
a $G$-torsor over $F$ and $G(F)=SO(V)$. Let $SO_\A(V)$ and
$GL_\A(V)$ be the adelic groups of $SO(V)$ and $GL(V)$ respectively.

An $\frak o_F$-lattice $L$ in $V$ is defined as a finitely generated
$\frak o_F$-module satisfying $FL=V$ and $L_\frak p$ is the local
completion of $L$ inside the local completion $V_\frak p$ of $V$ at
$\frak p$ for $\frak p< \infty_F$ and $L_\frak p=V_\frak p$ for
$\frak p\in \infty_F$.

\begin{lem}\label{act} Let $L$ be an $\frak o_F$-lattice in $V$ and $u_0\in V$.
For any $(\sigma_\frak p)_{\frak p\in \Omega_F}\in GL_\A(V)$, there
is a unique $\frak o_F$-lattice $L'$ in $V$ and $u'_0\in V$ unique
up to elements in $L'$ such that
$$L'_\frak p+u'_0=\sigma_\frak p(L_\frak p+u_0)$$ inside $V_\frak p$ for all $\frak p<\infty_F$.

\end{lem}

\begin{proof} By 81:14 of \cite{O73}, there is an $\frak o_F$-lattice $L'$ of
$V$ such that  $L'_\frak p=\sigma_\frak p L_\frak p$ for all $\frak
p< \infty_F$. By the strong approximation theorem for $\Bbb G_a$,
there is $u_0'\in V$ such that
$$\sigma_\frak p u_0 -u'_0 \in L'_\frak p $$ for all $\frak p<\infty_F$.
Therefore $$L'_\frak p+u'_0=L'_\frak p + \sigma_\frak p u_0 =
\sigma_\frak p L_\frak p + \sigma_\frak p u_0 =\sigma_\frak
p(L_\frak p+u_0)$$ for all $\frak p< \infty_F$ as required.

Suppose there is an $\frak o_F$-lattice $M$ in $V$ and $x_0\in V$
such that $$x_0+M_\frak p=u'_0+L'_\frak p $$ for all $\frak p <
\infty_F$. This implies that $M_\frak p=L'_\frak p$ and $x_0-u'_0
\in L'_\frak p$ for all $\frak p <\infty_F$. Therefore $M=L'$ and
$x_0-u'_0\in L'$ and the uniqueness follows.
\end{proof}

By Lemma \ref{act}, one can define the action of $GL_\A(V)$ on
$L+u_0$ and set
$$SO_\A(L+ u_0)=\{ \sigma \in SO_\A(V): \ \sigma
(L+ u_0)=(L+ u_0) \}. $$ By (\ref{loc}), $X_F(F)\neq \emptyset $ by
the Hasse principle. To fix a rational point $P$ in $X_F(F)$ is
equivalent to fix a non-zero vector $x_0\in V$ such that
$$B(x_0,x_0)=B(u_0,u_0)-g .$$

In particular, we can choose
$$L=\frak o_F v_1+\frak o_F v_2 \ \ \ \text{and} \ \ \
u_0=d^{-1}(2ce-bf)v_1+d^{-1}(2af-be)v_2 \in V. $$ If $\sigma \in
SO(V)$ with $\sigma x_0 \in (L+u_0)$, there are $\alpha, \beta \in
\frak o_F$ such that $\sigma x_0=\alpha v_1+\beta v_2+ u_0$. One can
verify that $(\alpha, \beta)$ is a solution of the equation
(\ref{equ}). This gives the map
$$ \{ \sigma \in SO(V): \  \sigma x_0\in (L+u_0)
\} \longrightarrow \bold X(\frak o_F) . $$ In fact, this map is
bijective by the Witt theorem (42:16 and 42:17 of \cite{O73}). One
can also extend this bijection to
$$\bold X(\frak o_{F_\frak p})\cong \{ \sigma \in SO(V_\frak p): \
x_0\in \sigma (L_\frak p+u_0) \} $$ for all $\frak p<\infty_F$.
Since $u_0\in L_\frak p$ and $\frak o_{F_\frak p}x_0$ splits
$L_\frak p$ for almost all $\frak p$, one has
$$SO_\A(L+u_0) \subseteq Stab_{\Bbb A}({\bf X})$$ with finite index.
In practice, we replace $Stab_{\Bbb A}({\bf X})$ by $SO_\A(L+u_0)$
which is much easier to be computed.

\bigskip

\subsection*{Split Cases} \label{split} $-d=b^2-4ac \in (F^\times)^2$.

Fix the isotropic basis of $V$ $$\{ (b+\varpi)v_1-2a v_2, \
(b-\varpi)v_1-2a v_2 \}$$ where $\varpi\in F^\times$ satisfying
$\varpi^2=b^2-4ac$. For any $F$-algebra $A$, there is a group
isomorphism $SO(V\otimes_FA)\cong A^\times$ by sending $(\sigma
\mapsto \varrho)$ such that
$$\sigma((b+\varpi)v_1-2a v_2)=\varrho[(b+\varpi)v_1-2a v_2]. $$  The map  gives
 $G\cong \Bbb G_m$. Such a case is not included in \S 1.
However this case is much simpler. Indeed one can factorize the
equation (\ref{equ}) into the following
$$\begin{aligned} & a[x+\frac{b+\varpi}{2a}y+\frac{1}{2}(\frac{e}{a}-\frac{2f}{\varpi}+\frac{eb}{a\varpi})]
[x+\frac{b-\varpi}{2a}y+\frac{1}{2}(\frac{e}{a}+\frac{2f}{\varpi}-\frac{eb}{a\varpi})]\\
=& (bef-ce^2-af^2)(b^2-4ac)^{-1}-g \end{aligned} $$ Define
$$f_F= 2ax+(b+\varpi)y+(e-\frac{2af}{\varpi}+\frac{eb}{\varpi}) \in
F[X_F]^\times $$ which induces the injection
$$f_F: \ \ \prod_{\frak p\in \Omega_F} \bold X(\frak o_{F_\frak p})
 \longrightarrow \Bbb I_F. $$

With the above identification, one has
$$\bold X(\frak o_F) \neq \emptyset \ \Leftrightarrow \
f_F[\prod_{\frak p\in \Omega_F} \bold X(\frak o_{F_\frak p})]\cap
[F^\times SO_\A(L+ u_0)] \neq \emptyset$$ by the same argument as
those in Proposition \ref{mainmodel}. By the class field theory and
Theorem 8.1 in \cite{PR94}, there is a finite abelian extension
$K_{L+u_0}/F$ such that the Artin map gives the isomorphism $$
\psi_{K_{L+u_0}/F}: \ \ \ \Bbb I_F/F^\times SO_\A(L+ u_0) \cong
Gal(K_{L+u_0}/F)
$$
The above criterion can be stated as $\bold X(\frak o_F)\neq
\emptyset$ if and only if there is
$$x_A\in\prod_{\frak p\in \Omega_F} \bold X(\frak o_{F_\frak p}) \ \ \ \text{such that} \ \ \
\psi_{K_{L+u_0}/F}(f_F(x_A))=1 $$ in $Gal(K_{L+u_0}/F)$. This
criterion does not need any knowledge of the fixed rational point
$P$ in $X_F(F)$.

The class field theory only guarantees the existence of such abelian
extension of $K_{L+u_0}/F$. In order to obtain the explicit
condition for $\bold X(\frak o_F)\neq \emptyset$, one needs further
explicit description of $K_{L+u_0}/F$. There are no general theory
for such explicit construction (Hilbert's 12-th problem) except the
cyclotomic field theory over $\Bbb Q$ and the complex multiplication
theory of elliptic curves over imaginary quadratic fields (see
\cite{Sh71}).

One can determine a finite set $S$ which contains all ramification
primes in $K_{L+u_0}/F$. Such a set $S$ can be chosen as all primes
containing $\infty_F$, all dyadic primes of $F$, all finite primes
such that $x_0$ or $u_0$ is not in $L_\frak p$, all finite primes
such that $B(x_0,x_0)$ or $B(u_0,u_0)$ is not in $\frak o_{F_\frak
p}^\times$ and all primes such that $L_\frak p$ is not unimodular.

\bigskip

\subsection*{Non-split Cases}\label{nsplit} $-d=b^2-4ac \not\in (F^\times)^2$.

Fix the orthogonal basis $\{ v_1, bv_1-2a v_2\}$ of $V$. For any
$F$-algebra $A$, there is a group isomorphism
$$SO(V\otimes_FA)\cong
ker [(A\otimes_F E)^\times  \xrightarrow{N_{E/F}} A^\times ]; \ \
\sigma \mapsto \alpha+\beta\sqrt{-d} $$ such that $\sigma v_1=\alpha
v_1+\beta (bv_1-2a v_2)$, where $E=F(\sqrt{-d})$. This gives $G\cong
R_{E/F}^1(\Bbb G_m)$. By this identification, one can obtain an open
subgroup $\Xi$ of $\Bbb I_E$ such that the map induced by inclusion
$$SO_\A(V)/SO(V)SO_\A(L+u_0) \longrightarrow {\Bbb I}_E/E^\times \Xi  $$
is injective by applying Theorem \ref{multiexsitence}.

Let$$f_E=
2ax+(b+\sqrt{-d})y+(e-\frac{2af}{\sqrt{-d}}+\frac{eb}{\sqrt{-d}})
\in E[X_F\times_F E]^\times $$ which induces the injection
$$f_E: \ \ \prod_{\frak p\in \Omega_F} \bold X(\frak o_{F_\frak p})
 \longrightarrow \Bbb I_E. $$

By the same argument as those in Proposition \ref{mainmodel}, one
has
$$\bold X(\frak o_F) \neq \emptyset  \ \Leftrightarrow \
f_E[\prod_{\frak p\in \Omega_F} \bold X(\frak o_{F_\frak p})]\cap
(E^\times \Xi) \neq \emptyset.$$ Let $K_{\Xi}/E$ be a finite abelian
extension such that $$\psi_{K_{\Xi}/E}: \ \ \ \Bbb I_E/E^\times \Xi
\cong Gal(K_{\Xi}/E) $$ by the Artin map. The above criterion can be
stated as $\bold X(\frak o_F)\neq \emptyset$ if and only if there is
\begin{equation} \label{criterion} x_A\in\prod_{\frak p\in \Omega_F}
\bold X(\frak o_{F_\frak p}) \ \ \ \text{such that} \ \ \
\psi_{K_\Xi/E}(f_E(x_A))=1 \end{equation} in $Gal(K_\Xi/E)$.

\begin{rem} \label{remark} The necessity of the above criterion
does not need that $\Xi$ is admissible. Therefore $\bold X(\frak
o_F)=\emptyset$ if there is an abelian extension $K/E$ such that
$$\psi_{K/E}(f_E(x_A)) \neq 1 \ \ \ \text{for any} \ \ \
x_A\in\prod_{\frak p\in \Omega_F} \bold X(\frak o_{F_\frak p})$$
where $\psi_{K/E}$ is the Artin map.
\end{rem}

For determining a finite set $S$ which contains all ramification
primes in $K_{\Xi}/F$, one needs to add $T_0$ in Lemma \ref{power}
as the part of $S$ besides the above primes of $F$ in split case.

\bigskip

\section{Examples with imaginary quadratic splitting fields}

In this section, we will provide some explicit examples with
imaginary quadratic splitting fields over $\Bbb Q$ to explain how to
apply the result in \S \ref{quadratic} to obtain the explicit
conditions for $\bold X(\Bbb Z)\neq \emptyset$. The explicit
construction of the abelian extensions of the splitting fields with
the required property is crucial. For the imaginary quadratic field
case, such abelian extensions can be given by the complex
multiplication of elliptic curves in principle (see Theorem 5.5 in
\cite{Sh71}). We will keep the same notation as that in \S
\ref{quadratic}.

The method described in \S \ref{quadratic} is also the natural
extension of Gauss' idea for determining the integral representation
of the positive integers by a positive definite binary quadratic
form. We explain this point by using one of the typical example
$x^2+dy^2$ with a positive integer $d$, which is studied in detail
in \cite{Co89}. Let $L=\Bbb Z+ \Bbb Z\sqrt{-d}$ be an order in
$E=\Bbb Q(\sqrt{-d})$.

\begin{prop} \label{special} Let $\bold X$ be the scheme over $\Bbb Z$ defined by
$x^2+dy^2=n$ for a positive integer $n$, and $K_L$ be the ring class
field corresponding to the order $L$. Then $\bold X(\Bbb Z)\neq
\emptyset$ if and only if there is $$(x_p,y_p)_{p\leq \infty}\in
\prod_{p\leq \infty} \bold X(\Bbb Z_p) \ \ \ \text{such that} \ \ \
\psi_{K_L/E}(\tilde f_E[\prod_{p\leq \infty}(x_p,y_p)])=1 $$ where
$\psi_{K_L/E}$ is the Artin map and
$$\tilde f_E[(x_p,y_p)]= \begin{cases} (x_p+y_p \sqrt{-d}, x_p-y_p\sqrt{-d}) \ \ \
& \text{if $p$ splits in $E/\Bbb Q$} \\
x_p+y_p\sqrt{-d} \ \ \ & \text{otherwise.}
\end{cases} $$
\end{prop}
\begin{proof}
Let $p$ be a prime and $L_p$ be the $p$-adic completion of $L$
inside $E_p=E\otimes_\Bbb Q \Bbb Q_p$. Since the ring class field
$K_L$ of the order $L$ corresponds to the open subgroup $E^\times (
\prod_{p\leq \infty} L_p^\times)$ of $\Bbb I_E$ by the class field
theory, one only needs to show that $\prod_{p\leq \infty}L_p^\times$
is a $L$-admissible subgroup of $\Bbb I_E$ with
$L_\infty^\times=\Bbb C^\times$ by (\ref{criterion}).

Since $xL_p=L_p$ for $x\in E_\frak p^\times$ if and only if $x\in
L_p^\times$, one has
$$SO(L_p)=\{ \alpha \in L_p^\times : \ N_{E_p/\Bbb
Q_p}(\alpha)=1 \} .$$ The natural group homomorphism
$$\lambda_E : \ \ \ SO_\A(V)/SO(V)SO_\A(L) \longrightarrow \Bbb
I_E/(E^\times  \prod_{p\leq \infty} L_p^\times)
$$ is well-defined.

Let $\alpha\in E^\times$ and $i\in \prod_{p\leq \infty} L_p^\times$
such that $\alpha i\in SO_\A(V)$. Then $$N_{E/\Bbb Q}(\alpha
i)=N_{E/\Bbb Q}(\alpha )N_{E/\Bbb Q}(i)=1 .$$  Since
$$N_{E/\Bbb Q}(\alpha) \in \Bbb Q\cap \prod_{p\leq \infty} \Z_p^\times =\{\pm 1\} \ \ \
\text{and}  \ \ \ N_{E/\Bbb Q}(\alpha)>0,$$ one concludes $N_{E/\Bbb
Q}( \alpha)=N_{E/\Bbb Q}(i)=1$. This implies that $$\alpha\in SO(V)
\ \ \ \text{and} \ \ \ i\in SO_\A(L) . $$ Therefore $\lambda_E$ is
injective and the proof is complete. \end{proof}

One can recover the following classical result (see Theorem 9.4 in
\cite{Co89}) which is obvious when $L=\Bbb Z+\Bbb Z\sqrt{-d}$ is the
ring of integers of $E$ and $K_L$ is the Hilbert class field of $E$.
In general, for example if $d$ is not square free, one needs to
study the class groups of orders which are not Dedekind domains (see
\cite{Co89}).

\begin{cor} Let $l$ be an odd prime not dividing $d$. The equation $ x^2+dy^2=l$
is solvable over $\Bbb Z$ if and only if $l$ splits completely in
$K_L$.
\end{cor}
\begin{proof} Let $\bold X$ be the scheme over $\Bbb Z$ defined by the equation.
By a simple computation, $\prod_{p\leq \infty} \bold X(\Bbb Z_p)\neq
\emptyset$ if and only if $l$ splits completely in $E$ and
$$
\begin{cases}
(\frac{l}{q})=1 \ \ \ & \text{for each odd prime $q|d$} \\
l\equiv 1 \  \text{or} \  d  \mod 4 \ \ \ & \text{if $d$ is odd} \\
l\equiv 1 \ \text{or} \ d+1 \mod 8 \ \ \ & \text{if \ $d \equiv 2
\mod
4$} \\
l\equiv 1 \mod 4 \ \ \ & \text{if \ $d\equiv 4 \mod 8$} \\
l\equiv 1 \mod 8 \ \ \ & \text{if $8\mid d$}.
 \end{cases}$$

First we show that the above local conditions can be obtained when
$l$ splits completely in $K_L$. Indeed, for any odd prime $q|d$,
$E(\sqrt{q^*})/E$ is unramified for all primes except the primes
above $q$, where $q^*=(\frac{-1}{q})q$. To verify if $\sqrt{q^*} \in
K_L$, one only needs to see if $\psi (x)(\sqrt{q^*})=\sqrt{q^*}$ for
all $x\in \prod_{p\leq \infty} L_p^\times$, where $\psi$ is the
Artin map over $E$. This is equivalent to check if the product of
the Hilbert symbols $\prod_{\frak P}(q^*,x_\frak P)_\frak P=1$,
where $\frak P$ runs over all primes of $E$ and $x=(x_\frak P)\in
\prod_{p\leq \infty} L_p^\times$. It is clear that $(q^*, x_\frak
P)_\frak P=1$ for $\frak P\nmid q$.

Since $x=(x_\frak P)\in \prod_{p\leq \infty} L_p^\times$, there are
$a_p, b_p\in \Z_p$ such that
$$ \begin{cases} (x_\frak P, x_{\bar{\frak P}})=(a_p+b_p \sqrt{-d}, \ a_p-b_p
\sqrt{-d}) \ \ \ & \text{ if $p=\frak P\bar{\frak P}$ splits in
$E/\Q$ }
\\  x_\frak
P= a_p+b_p \sqrt{-d}  \ \ \ & \text{otherwise .}
\end{cases} $$ Hence
$$\prod_{\frak P\mid q} (q^*, x_\frak P)_\frak P=(q^*, a_q^2-b_q^2 d)_q=1 $$ by
(1.5.3) Proposition, (7.1.4) Corollary in \cite{NSW} and $q|d$.
Therefore $\sqrt{q^*} \in K_L$. This implies that $l$ splits in
$\Bbb Q(\sqrt{q^*})$ and $(\frac{q^*}{l})=(\frac{l}{q})=1$. This
means that $l$ is a square in $\Z_q$ for any odd prime $q|d$. By the
Hilbert reciprocity law, one has
$$(-d,l)_2=\prod_{\text{odd primes $q$}}(-d,l)_q=\prod_{q|dl} (-d,l)_q = 1. $$ One gets the local
conditions for odd $d$ and $d \equiv 2 \mod 4$ by $(-d,l)_2=1$.

If $d \equiv 4 \mod 8$, then $\Bbb Q(\sqrt{-1})\subset K_L$ by the
same argument as above. So $l\equiv 1 \mod 4$.

If $8\mid d$, then $\Bbb Q(\sqrt{-1},\sqrt{2}) \subset K_L$. So
$l\equiv 1 \mod 8$.

Let $(x_p,y_p)\in  \bold X(\Bbb Z_p)$ for $p\leq \infty$. Then
$\tilde f_E[(x_p,y_p)]\in L_p^\times$ for $p\neq l$. Since $l$
splits in $E/\Bbb Q$ by the local conditions and
$$(x_l+y_l\sqrt{-d})(x_l-y_l\sqrt{-d})=l, $$
one can write $\frak L_1$ and $\frak L_2$ as two primes in $E$ above
$l$ and assume $\tilde f_E[(x_l,y_l)]$ is a prime element at
$E_{\frak L_1}$ and a unit at $E_{\frak L_2}$. By the class field
theory, one concludes that $\frak L_1$ splits in $K_L/E$ if and only
if
$$\psi_{K_L/E}(\tilde f_E[\prod_{p\leq \infty}(x_p,y_p)])=1.$$

If $l$ splits completely in $K_L/\Bbb Q$, then $\bold X(\Bbb Z)\neq
\emptyset$ by Proposition \ref{special}.

Conversely, choose two integral solutions $(x_0,y_0)$ and
$(x_0,-y_0)$ in ${\bf X}(\Z)$. One obtains that both $\frak L_1$ and
$\frak L_2$ split in $K_L/E$ by the above argument. The proof is
complete.
\end{proof}

We give one more example to explain that the method provided in the
previous sections is beyond Gauss' method which only applies to the
binary quadratic forms.

For any positive integer $n$, one can write $n=2^s{p_1}^{e_1}\cdots
{p_g}^{e_g}$ and define
$$\aligned & D(n)=\{p_1, \cdots, p_g \} \cr
& D_1= \{ p\in D(n) : \ (\frac{-1}{p})=-(\frac{2}{p})=1  \} \cr &
D_2=\{p\in D(n) : \ (\frac{-1}{p})=(\frac{2}{p})=1 \text{ and }
x^4\equiv 2 \ mod\ p \ \text{ is not solvable} \}.
\endaligned$$

\begin{exa} Let $n$ be a positive integer with the above notation.
Then the equation $$x^2+64y^2+64y+16=n$$ is solvable over $\Bbb Z$
if and only if $p_j\equiv 1 \mod 4$ for odd $e_j$ with $1\leq j\leq
g$ and one of the following conditions holds

(1) $s=0$ and $n\equiv 1 \mod 8$ and $$(D_1\neq \emptyset \ \ \
\text{or} \ \ \ \sum_{p_j\in D_2}e_j \equiv 1 \mod 2 \ \text{for} \
D_1=\emptyset) .$$

(2) $s=2$ and $2^{-2}n \equiv 5 \mod 8$.

(3) $s= 4, 5$.

\end{exa}
\begin{proof} Let $\bold X$ be the scheme over $\Bbb Z$ defined by the equation.
By a simple computation, the necessary and sufficient conditions for
$\prod_{p\leq \infty} \bold X(\Bbb Z_p)\neq \emptyset$ are
$p_j\equiv 1 \mod 4$ for odd $e_j$ with $1\leq j\leq g$
and $$\begin{cases} s=0 \ \ \ \text{and} \ \ \ n\equiv 1 \mod 8 \\
s=2 \ \ \ \text{and} \ \ \ 2^{-2}n \equiv 5 \mod 8 \\
s=4,5. \end{cases}$$

Let $E=\Bbb Q(\sqrt{-1})$ and $L=\Bbb Z+\Bbb Z 8\sqrt{-1}$ be an
order in $E$ and $u_0=4\sqrt{-1}$. For any prime $p$,
$E_p=E\otimes_\Bbb Q \Bbb Q_p$ and $L_p$ is the $p$-adic completion
of $L$ inside $E_p$. Write $E_\infty=L_\infty=\Bbb C$. Then
$$\lambda_E[SO_\A(L+ u_0)]\subseteq \prod_{p\leq \infty} L_p^\times$$ and
the induced map
$$\lambda_E : \ \ \ SO_\A(V)/SO(V)SO_\A(L+ u_0) \longrightarrow \Bbb
I_E/E^\times  ( \prod_{p\leq \infty} L_p^\times)
$$ is injective by the same argument in Proposition \ref{special}.
This implies that $\prod_{p\leq \infty} L_p^\times$ is $(L+
u_0)$-admissible. The abelian extension $K_L$ of $E$ corresponding
to $E^\times \prod_{p\leq \infty} L_p^\times$ is the ring class
field of $L$. By Proposition 9.5 in \cite{Co89},
$K_L=E(\root{4}\of{2})$. Then $K_L/E$ is unramified at all places
except $2$ and $2$ is totally ramified in $K_L/\Bbb Q$ and
$Gal(K_L/E)=\mu_4$ where $\mu_4$ is the set of all $4$-th roots of
unity.

Since
 $p\in L_2^\times$ for any odd prime $p$, one concludes that
 \begin{equation} \label{p} 1=\begin{cases} \psi_{K_L/E} (p_\frak P) \ \ \ & \text{if $p$ is inert in
$E/\Bbb Q$
with $\frak P\mid p$} \\
\psi_{K_L/E}(p_{\frak P_1}) \psi_{K_L/E} (p_{\frak P_2}) \ \ \ &
\text{if $p=\frak P\bar{\frak P}$ in $E/\Bbb Q$}
\end{cases} \end{equation}
where $p_\frak P$ (resp. $p_{\bar{\frak P}}$) is in $\Bbb I_E$ such
that its $\frak P$ (resp. $\bar{\frak P}$) component is $p$ and the
other components are 1 and $\psi_{K_L/E}$ is the Artin symbol.

Define
$$\tilde f_E[(x_p,y_p)]= \begin{cases} (x_p+4(2y_p+1) \sqrt{-1}, x_p-4(2y_p+1)\sqrt{-1}) \ \ \
& \text{if $p$ splits in $E/\Bbb Q$} \\
x_p+4(2y_p+1)\sqrt{-1} \ \ \ & \text{otherwise}
\end{cases} $$
for any $(x_p,y_p)\in \bold X(\Bbb Z_p)$.

If $(p,2n)=1$, then
$$\tilde f_E[(x_p,y_p)] \in L_p^\times  \ \ \ \text{and} \ \ \
\psi_{K_L/E}(\tilde f_E[(x_p,y_p)])=1 ,$$ where $\tilde
f_E[(x_p,y_p)]$ is also regarded as an element in $\Bbb I_E$ such
that the components above $p$ are given by the value of $\tilde
f_E[(x_p,y_p)]$ and 1 otherwise.

If $p$ is inert in $E/\Bbb Q$ with $p\in D(n)$, then
$\psi_{K_L/E}(\tilde f_E[(x_p,y_p)])=1$ by (\ref{p}).

If $p$ splits in $E/\Bbb Q$ with $p\in D(n)$, one has
$$ \begin{cases} \psi_{K_L/E}(p_{\frak P}) =\psi_{K_L/E} (p_{\bar{\frak P}})=-1  \ \ \ & \text{for $p\in D_2$} \\
\psi_{K_L/E}(p_{\frak P}) =-\psi_{K_L/E} (p_{\bar{\frak P}})=\pm
\sqrt{-1} \ \ \ & \text{for $p\in D_1$  }
\end{cases}  $$ by the definition of $D_1$, $D_2$ and $(\ref{p})$.  Since
$$(x_p+4(2y_p+1) \sqrt{-1})(x_p-4(2y_p+1)\sqrt{-1})=n , $$ one has
\begin{equation} \label{symbol}
\psi_{K_L/E}(\tilde f_E[(x_p,y_p)])=\begin{cases} (-1)^{e} \ \ \ &
\text{if $p\in D_2$} \cr (-1)^{a}(\pm \sqrt{-1})^{e} \ \ \ &
\text{if $p\in D_1$} \cr 1 \ \ \ & \text{otherwise}
\end{cases}
\end{equation}
where $a=ord_{\frak P}(x_{p}+4(2y_{p}+1)\sqrt{-1})$ and $e$ is the
exponent of $p$ inside $n$. By the Hensel's lemma, there are two
local solutions $(x_{p},y_{p})\in \bold X(\Bbb Z_{p})$ such that
$a=0$ and $1$ respectively. Therefore one can rewrite (\ref{symbol})
for $p\in D_1$ as
\begin{equation} \label{wish} \psi_{K_L/E}(\tilde f_E[(x_{p},y_{p})])=\begin{cases} \pm \sqrt{-1} \ \ \ & \text{ $e$ is
odd} \cr \pm 1 \ \ \ & \text{ $e$ is even} \end{cases}
\end{equation} where the sign can be chosen as one wishes.

Suppose $s=4$ or $5$. Since $x^2+(2y+1)^2=2^{-4}n$ is solvable over
$\Bbb Z$ by the above local conditions, one concludes that
$x^2+16(2y+1)^2=n$ is solvable over $\Bbb Z$.

Suppose $s=2$. Then the local solution $(x_2,y_2)\in \bold X(\Bbb
Z_2)$ satisfies $x_2\equiv 2 \mod 4$. Let $v$ be the unique place of
$E$ over $2$.  Since
$$(2,x_2+4(2y_2+1)\sqrt{-1})_v=(2,2)_v\cdot (2,2^{-1}x_2+2(2y_2+1)\sqrt{-1})_v=(2,5)_2=-1$$
by (1.5.3) Proposition and (7.1.4) Corollary in \cite{NSW}, this
implies that
$$\psi_{K_L/E}(\tilde f_E[(x_2,y_2)])=\pm \sqrt{-1} .$$
Since $2^{-2}n \equiv 5 \mod 8$, the number of primes $p_j$ in $D_1$
with odd $e_j$ is odd. Therefore there are an even number of local
primes $p$ such that $$\psi_{K_L/E}(\tilde f_E[(x_p,y_p)])=\pm
\sqrt{-1}$$ by the above computation and (\ref{wish}). By choosing
the sign properly, one concludes that there is
$$(x_p,y_p)_{p\leq \infty}\in \prod_{p\leq \infty} \bold X(\Bbb
Z_p) \ \ \ \text{such that} \ \ \ \psi_{K_L/E}(\tilde
f_E[\prod_{p\leq \infty}(x_p,y_p)])=1.
$$ Therefore $\bold X(\Bbb Z)\neq \emptyset$ by (\ref{criterion}).

Suppose $s=0$. Then the local solution $(x_2,y_2)\in \bold X(\Bbb
Z_2)$ satisfies $x_2\equiv 1 \mod 2$. Let $v$ be the unique place of
$E$ over $2$. Since
$$\tilde f_E[(x_2,y_2)]=x_2+4(2y_2+1)\sqrt{-1}\equiv x_2+4\sqrt{-1}\equiv 1+4\sqrt{-1} \mod
L_2^\times, $$  and $(1+4\sqrt{-1})\in E^\times$ with
$(1+4\sqrt{-1})(1-4\sqrt{-1})=17$, one concludes that
$\psi_{K_L/E}(\tilde f_E[(x_2,y_2)])$ is equal to the Frobenius of a
prime above $17$ in $Gal(K_L/E)$ which is $-1$.

If $D_1\neq \emptyset$, one has that either all $e_j$ for $p_j\in
D_1$ are even or the number of primes $p_j$ in $D_1$ with odd $e_j$
is even by using $n\equiv 1 \mod 8$. By (\ref{wish}) and choosing
the sign properly, there is $$(x_p,y_p)_{p\leq \infty}\in
\prod_{p\leq \infty} \bold X(\Bbb Z_p) \ \ \ \text{such that} \ \ \
\psi_{K_L/E}(\tilde f_E[\prod_{p\leq \infty}(x_p,y_p)])=1. $$
Therefore $\bold X(\Bbb Z)\neq \emptyset$ by (\ref{criterion}).

If $D_1=\emptyset$, then  $$ \psi_{K_L/E}(\tilde f_E[\prod_{p\leq
\infty}(x_p,y_p)])=(-1)^{1+\sum_{p_j\in D_2}e_j}$$ by (\ref{symbol})
for any $(x_p,y_p)_{p\leq \infty}\in \prod_{p\leq \infty} \bold
X(\Bbb Z_p)$. By (\ref{criterion}), one concludes $$\bold X(\Bbb
Z)\neq \emptyset \ \ \  \Leftrightarrow \ \ \ \sum_{p_j\in D_2}e_j
\equiv 1 \mod 2 $$ and the proof is complete.
\end{proof}

\bigskip

\section{Examples with real quadratic splitting fields}

Gauss' method can not be applied to integral representations of
indefinite binary quadratic forms in general either because the
norms of the elements of real quadratic fields are not always
positive and one can not decide the sign effectively. The typical
example is to determine exactly when the negative Pell equation
$x^2-dy^2=-1$ is solvable over $\Z$ for a positive integer $d$.

In order to apply the method in \S \ref{quadratic}, for example, to
determine all integers represented by $x^2-dy^2$ for a positive
integer $d$ explicitly, one can take the order $L=\Bbb Z+ \Bbb Z
\sqrt{d}$ inside $E=\Bbb Q(\sqrt{d})$ and needs to determine the
abelian extension of $E$ corresponding to an $L$-admissible subgroup
of $\Bbb I_E$ explicitly. There is no general theory for obtaining
such abelian extensions for real quadratic fields but ad hoc method.
We will keep the same notation as that in \S \ref{quadratic}.

With some extra conditions, one still has the analogous result to
Proposition \ref{special} for real quadratic fields.

\begin{prop} \label{real} Let $\bold X$ be the scheme defined by
$x^2-dy^2=n$ for some integer $n$ and $K_L$ be the ring class field
corresponding to the order $L$. Suppose one of the following
conditions holds:

(1) The equation $x^2-dy^2=-1$ has an integer solution.

(2) The equation $x^2-dy^2=-1$ has no local integral solutions at a
prime.

Then $\bold X(\Bbb Z)\neq \emptyset$ if and only if there is
$$(x_p,y_p)_{p\leq \infty}\in \prod_{p\leq \infty} \bold X(\Bbb
Z_p) \ \ \ \text{such that} \ \ \ \psi_{K_L/E}(\tilde
f_E[\prod_{p\leq \infty}(x_p,y_p)])=1 $$ where $\psi_{K_L/E}$ is the
Artin map and
$$\tilde f_E[(x_p,y_p)]= \begin{cases} (x_p+y_p \sqrt{d}, x_p-y_p\sqrt{d}) \ \ \
& \text{if $p$ splits in $E/\Bbb Q$} \\
x_p+y_p\sqrt{d} \ \ \ & \text{otherwise.}
\end{cases} $$
\end{prop}
\begin{proof} The same argument as that in Proposition
\ref{special} with the crucial step to show that $\lambda_E$ is
injective is available except $N_{E/\Q}(\alpha)>0$. Suppose
$N_{E/\Q}(\alpha)=-1$. Then $N_{E/\Q}(i)=-1$. This contradicts
Condition (2). If Condition (1) holds, there are $x_0,y_0\in \Z$
such that $x_0^2-dy_0^2=-1$. One can replace $\alpha$ by
$\alpha(x_0-y_0\sqrt{d})$ and $i$ by $i (x_0-y_0\sqrt{d})^{-1}$. The
proof is complete.
\end{proof}

When $d$ is a prime, the condition (1) or (2) in Proposition
\ref{real} will be satisfied. For $d=2l$ with a prime $l$, the
condition (2) in Proposition \ref{real} holds at prime $2$ if
$l\equiv 3 \mod 4$ and the condition (1) in Proposition \ref{real}
holds if $l\equiv 5 \mod 8$ by \cite{D}. Such cases can also be
treated by using the Gauss' method.

However the Gauss' method only uses the ring class fields to detect
which ideals are principal. Now we consider the case that one has to
use the abelian extensions beyond the ring class fields. Let $l$ be
a prime with $l\equiv 1 \mod 8$. Fix an integral solution
$(x_0,y_0,z_0)$ of the equation $x^2-2ly^2=2z^2$ such that $x_0>0$
and $(x_0,y_0,z_0)=1$
 by the Hasse principle. Let
$\Theta=E(\sqrt{x_0-y_0\sqrt{2l}})$. Then $\Theta$ is totally real
and $\Theta/E$ is unramified over all primes except the prime above
$2$ and $2$ is totally ramified in $\Theta/\Bbb Q$. This $\Theta$
will play role for solving the equation $x^2-2l y^2=n$ over $\Bbb
Z$.

\begin{lem} \label{computation} Let $l$ be a prime with $l\equiv 1 \mod 8$
satisfying $2l=r^2+s^2$ for two integers $r$ and $s$ with $s \equiv
\pm 3 \mod 8$. If $x_2$ and $y_2$ in $\Q_2$ satisfy
$x_2^2-2ly_2^2=-1$, then the Hilbert symbol
$$(x_2+y_2\sqrt{2l}, x_0-y_0\sqrt{2l})_v=-1$$
where $v$ is the unique prime in $E$ above $2$ and $(x_0, y_0)$ is
given as above.
\end{lem}
\begin{proof} For any $\xi\in E_v^\times $ with $N_{L_v/\Bbb Q_2}(\xi)=1$, there
exists $\beta\in E_v^\times$ such that $\xi=\sigma(\beta)\beta^{-1}$
by Hilbert 90, where $\sigma$ is the non-trivial element in
$Gal(E_v/\Bbb Q_2)$. Then
$$\aligned & (\xi, x_0-y_0\sqrt{2l})_v=(\sigma(\beta)\beta^{-1},
x_0-y_0\sqrt{2l})_v \cr  = \ & (N_{E_v/\Bbb Q_2}(\beta),
x_0-y_0\sqrt{2l})_v=(N_{E_v/\Bbb Q_2}(\beta),2)_2 \endaligned $$ by
(1.5.3) Proposition and (7.1.4) Corollary in \cite{NSW}. Since
$l\equiv 1 \mod 8$ and $E_v=\Bbb Q_2(\sqrt{2l})$, one has
$$(\xi, x_0-y_0\sqrt{2l})_v =(N_{E_v/\Bbb Q_2}(\beta),2)_2=(N_{E_v/\Bbb Q_2}(\beta),2l)_2=1 .$$
Therefore
$$\aligned & (x_2+y_2\sqrt{2l}, x_0-y_0\sqrt{2l})_v=((r-\sqrt{2l})s^{-1}, x_0-y_0\sqrt{2l})_v \\
       = \ & (r-\sqrt{2l}, x_0-y_0\sqrt{2l})_v (s, x_0-y_0\sqrt{2l})_v  \\
       = \ & (r-\sqrt{2l}, x_0-y_0\sqrt{2l})_v(s,2z_0^2)_2 = -(r-\sqrt{2l}, x_0-y_0\sqrt{2l})_v. \endaligned  $$
By the Hilbert reciprocity law, one has
$$(r-\sqrt{2l}, x_0-y_0\sqrt{2l})_v=\prod_{\frak p\neq v}(r-\sqrt{2l}, x_0-y_0\sqrt{2l})_{\frak p}.$$
Since $$(r-\sqrt{2l})(r+\sqrt{2l})=-s^2 \ \ \ \text{and} \ \ \
(x_0-y_0\sqrt{2l})(x_0+y_0\sqrt{2l})=2z_0^2$$ with $(r,s)=1$ and
$(x_0,y_0,z_0)=1$ respectively, one has $$ord_\frak
p(r-\sqrt{2l})\equiv ord_\frak p(x_0-y_0\sqrt{2l})\equiv 0 \mod 2
$$ for $\frak p\neq v$ and $\frak p<\infty_E$. Since
$x_0-y_0\sqrt{2l}>0$ over $\frak p\in \infty_E$, one obtains that
$$(r-\sqrt{2l}, x_0-y_0\sqrt{2l})_{\frak p}=1 $$ for all $\frak p\neq
v$. One concludes that $(x_2+y_2\sqrt{2l}, x_0-y_0\sqrt{2l})_v=-1$.
\end{proof}

The immediate corollary is the following result which was first
proved in \cite{Ep} by using the continued fractional method and was
reproved in \cite{Re} by using the reciprocity law in a quite
complicated way.

\begin{cor} \label{er} [Epstein - R\'{e}dei] If $l$ is a prime with $l\equiv 1 \mod 8$
satisfying $2l=r^2+s^2$ for two integers $r$ and $s$ with $s \equiv
\pm 3 \mod 8$, then the equation $x^2-2ly^2=-1$ is not solvable over
$\Bbb Z$.
\end{cor}
\begin{proof} Let $\bold X$ be the scheme defined by
$x^2-2ly^2=-1$. Then $$ \psi_{\Theta/E}(\tilde f_E[(x_p,y_p)_{p\leq
\infty}])=(x_2+y_2\sqrt{2l}, x_0-y_0\sqrt{2l})_v=-1
$$ for any $\prod_{p\leq \infty}(x_p,y_p)\in \prod_{p\leq \infty}
\bold X(\Bbb Z_p)$ by Lemma \ref{computation}. The result follows
from Remark \ref{remark}.
\end{proof}

It is natural to ask how to decide the solvability of the above
equation if one replaces $-1$ by an integer $n$.

\begin{prop}\label{extra} Let $l$ be a prime as in Lemma \ref{computation} and $\bold X$ be the scheme
over $\Bbb Z$ defined by $x^2-2ly^2=n$ for some integer $n$. Then
$\bold X(\Bbb Z)\neq \emptyset$ if and only if there is
$(x_p,y_p)_{p\leq \infty}\in \prod_{p\leq \infty} \bold X(\Bbb Z_p)$
such that
$$\psi_{H/E}(\tilde f_E[(x_p,y_p)_{p\leq \infty}])=1 \ \ \
\text{and} \ \ \ \psi_{\Theta/E}(\tilde f_E[(x_p,y_p)_{p\leq
\infty}])=1 $$ where $H$ is the Hilbert class field of $E$, $\Theta$
is the quadratic extension of $E$ defined as above, both
$\psi_{H/E}$ and $\psi_{\Theta/E}$ are Artin maps and $$\tilde
f_E[(x_p,y_p)]=
\begin{cases} (x_p+y_p \sqrt{2l}, x_p-y_p\sqrt{2l}) \ \ \
& \text{if $p$ splits in $E/\Bbb Q$} \\
x_p+y_p\sqrt{2l} \ \ \ & \text{otherwise.}
\end{cases} $$

\end{prop}
\begin{proof}
Let $L=\Bbb Z + \Bbb Z \sqrt{2l}$ be the maximal order of $E$. For
any prime $p$, $E_p=E\otimes_\Bbb Q \Bbb Q_p$ and $L_p$ is the
$p$-adic completion of $L$ inside $E_p$. Write
$E_\infty=L_\infty=\Bbb R$.

Let $v$ be the unique prime of $E$ above $2$. By the computation in
Lemma \ref{computation}, one has $(\xi, x_0-y_0\sqrt{2l})_v=1$ for
any $\xi\in E_v^\times $ with $N_{E_v/\Bbb Q_2}(\xi)=1$. This
implies that
$$\lambda_E(SO(L_2))\subseteq N_{\Theta_\frak V/E_v}(\Theta_\frak
V^\times) $$ where $\frak V$ is the unique prime of $\Theta$ above
$v$. Since $\Theta/E$ is unramified over all primes except $v$, one
concludes that $\lambda_E(SO_\A(L))\subseteq N_{\Theta/E}(\Bbb
I_{\Theta})$. Therefore the natural group homomorphism
$$\lambda_E : \ SO_\A(V)/SO(V)SO_\A(L)\longrightarrow [\Bbb I_E/(E^\times  N_{\Theta/E}(\Bbb
I_{\Theta}))] \times [\Bbb I_E/(E^\times \prod_{p\leq \infty}
L_p^\times)]$$ is well-defined.

Let $u\in ker \lambda_E$. Then there are $\alpha\in E^\times$ and
$i\in \prod_{p\leq \infty} L_p^\times$ with $\lambda_E(u)=\alpha i$.
Therefore
$$N_{E/\Bbb Q}(\alpha)=N_{E/\Bbb Q}(i)^{-1} \in \Bbb Q^\times \cap
(\prod_{p\leq \infty}\Z_p^\times )=\{\pm 1\} . $$

Suppose $N_{E/\Bbb Q}(\alpha)\neq 1$. Then $N_{E/\Bbb
Q}(\alpha)=N_{E/\Bbb Q}(i)=-1$. Write $i=(i_\frak p)_{\frak p}\in
\Bbb I_E$. Since $\Theta/E$ is unramified over all primes of $E$
except $v$, one has $\psi_{\Theta/E}(i_\frak p)$ is trivial for all
primes $\frak p\neq v$, where $i_\frak p$ is regarded as an idele
whose $\frak p$-component is $i_\frak p$ and 1 otherwise. Since
$N_{E/\Q}(i_v)=N_{E_v/\Q_2}(i_v)=-1$, one gets
$$\psi_{\Theta/E}(\alpha i)=\psi_{\Theta/E}(i)=\psi_{\Theta/E}(i_v)=-1$$
by Lemma \ref{computation}. This contradicts $u \in ker \lambda_E$.

Therefore $N_{E/\Bbb Q}(\alpha)=N_{E/\Bbb Q}(i)=1$. This implies
that $\alpha\in SO(V)$ and $i\in SO_\A(L)$. One concludes that
 $\lambda_E$ is injective. The
result follows from Corollary \ref{efmulti}.
\end{proof}

Finally we will use Proposition \ref{extra} to give an explicit
example. For any integer $n$, one can write $n=(-1)^{s_0}
2^{s_1}17^{s_2}{p_1}^{e_1}\cdots {p_g}^{e_g}$ and $D(n)=\{p_1,
\cdots, p_g \}$. Decompose $D(n)$ into the disjoint union of the
following subsets
$$\aligned & D_1=\{p\in D(n) : \ (\frac{2}{p})=(\frac{17}{p})=-1  \}  \ \ \text{and} \ \ D_2=\{p\in D(n): \
(\frac{34}{p})=-1 \} \cr
 & D_3= \{p\in D(n): \
(\frac{2}{p})=(\frac{17}{p})=1 \ \text{and} \ x^4-12x^2+2 \equiv 0 \
mod \ p \ \ \text{is solvable} \} \cr &  D_4= \{p\in D(n): \
(\frac{2}{p})=(\frac{17}{p})=1 \ \text{and}\ x^4-12x^2+2 \equiv 0 \
mod \ p \ \ \text{is not solvable} \} . \endaligned $$ Let
$$n_1=(-1)^{s_0} \prod_{p_i\in D(n)\setminus D_2} p_i^{e_i} .$$

\begin{exa} \label{pell} Let $n$ be an integer with the above notation.
Then the equation $$x^2-34y^2=n$$ is solvable over $\Bbb Z$ if and
only if

(1)  $n_1\equiv \pm 1 \mod 8$, $(\frac{n_1}{17})=1$ and
$(\frac{34}{p_i})=1$ for odd $e_i$

(2)  $D_1\neq \emptyset$; or $$\sum_{p_i\in D_4} e_i \equiv
\begin{cases} s_2 \mod 2 \ \ \ & \text{if $n_1\equiv 1, -9 \mod 16$} \cr
 s_2+1 \mod 2 \ \ \ & \text{if $n_1\equiv -1, 9 \mod 16$} \end{cases} $$ for $D_1=\emptyset$.
\end{exa}

\begin{proof} Let $\bold X$ be the scheme over $\Bbb Z$ defined by the equation.
By \cite{O58} for odd primes $p$ and \cite{Ko73} for $p=2$, one has
the condition (1) is equivalent to $\prod_{p\leq \infty} \bold
X(\Bbb Z_p)\neq \emptyset$.

Since $34=5^2+3^2$, Proposition \ref{extra} can be applied. Moreover
the Hilbert class field $H$ of $E=\Bbb Q(\sqrt{34})$ is $\Bbb
Q(\sqrt{2},\sqrt{17})$. Since the equation $x^2-34y^2=2$ has an
integral solution for $x=6$ and $y=1$, one can choose
$\Theta=E(\sqrt{6-\sqrt{34}})$. For simplicity, we will denote both
$Gal(\Theta/E)$ and $Gal(H/E)$ as $\mu_2=\{\pm 1 \}$.

Define $$\tilde f_E[(x_p,y_p)]= \begin{cases} (x_p+y_p \sqrt{34},
x_p-y_p\sqrt{34}) \ \ \
& \text{if $p$ splits in $E/\Bbb Q$} \\
x_p+y_p\sqrt{34} \ \ \ & \text{otherwise}
\end{cases} $$
for any $(x_p,y_p)\in \bold X(\Bbb Z_p)$.

Then
$$\psi_{H/E}(\tilde f_E[(x_p,y_p)])= \begin{cases}  (-1)^{e_i}  \ \ \ & \text{if
$p\in D_1$} \cr 1 \ \ \ & \text{otherwise.}
\end{cases} $$ Since $n_1\equiv \pm 1 \mod 8$ by
the local condition (1), one always has $\sum_{p_i\in D_1}e_i \equiv
0 \mod 2$. This implies that
$$\psi_{H/E}(\tilde f_E[(x_p,y_p)_{p\leq \infty}])=1 \ \ \ \text{for
any  \ $(x_p,y_p)_{p\leq \infty}\in \prod_{p\leq \infty} \bold
X(\Bbb Z_p)$}.
$$

Next we compute $\psi_{\Theta/E}$. It is clear that
$\psi_{\Theta/E}(\tilde f_E[(x_p,y_p)])=1$ for $p\nmid 34n$.

If $p\in D_1$, then $p=\frak P\bar{\frak P}$ over $E$. Moreover, one
of $\frak P$ and $\bar{\frak P}$ splits in $\Theta/E$ and the other
is inert in $\Theta/E$. Without loss of generality, we assume that
$\frak P$ is inert in $\Theta/E$. Then
$$\psi_{\Theta/E}(\tilde f_E[(x_{p},y_{p})])=(-1)^{a}  \ \ \
\text{with} \ \ \ a=ord_{\frak P}(x_{p}+y_{p}\sqrt{34}). $$

If $p\in D_3$, then $p$ splits completely in $\Theta/\Q$ and
$\psi_{\Theta/E}(\tilde f_E[(x_{p},y_{p})])=1$.

If $p\in D_4$, then $p=\frak P\bar{\frak P}$ over $E$ and both
$\frak P$ and $\bar{\frak P}$ are inert in $\Theta/E$. One gets
$$\psi_{\Theta/E}(\tilde f_E[(x_p,y_p)])=(-1)^e$$ where $e$ is the
exponent of $p$ inside $n$.

If $p\in D_2$, the exponent $e$ of $p$ inside $n$ satisfies $e\equiv
0 \mod 2$ by the condition (1) and $p$ is inert in $E/\Q$. Moreover,
the prime $\frak P$ above $p$ in $E$ splits in $\Theta/E$ if and
only if $6-\sqrt{34}$ is a square in $E_\frak P$. By the Hensel's
lemma, this is equivalent to $(6-\sqrt{34}, p)_\frak P=(2,p)_p=1$.
Since $ord_\frak P(x_p+y_p \sqrt{34})=\frac{1}{2}ord_p(n)$, one has
$$\psi_{\Theta/E}(\tilde
f_E[(x_p,y_p)])=(\frac{2}{p})^{\frac{1}{2}e}. $$

We summarize
$$\psi_{\Theta/E}(\tilde f_E[(x_p,y_p)])= \begin{cases} 1 \ \ \ & \text{if
$(p,34n)=1$ or $p\in D_3$} \cr (-1)^{s_2} \ \ \ & \text{if $p=17$}
\cr
 (-1)^{a}  \ \ \ & \text{if
$p\in D_1$} \cr (\frac{2}{p})^{\frac{1}{2}e} \ \ \ & \text{if $p\in
D_2$} \cr (-1)^{e} \ \ \ & \text{if $p\in D_4$}
\end{cases} $$ for any $(x_p,y_p)\in \bold X(\Bbb Z_p)$.

When $D_1\neq \emptyset$, there are two local solutions
$(x_{p},y_{p})\in \bold X(\Bbb Z_{p})$ such that $a=0$ and $1$
respectively for any $p\in D_1$ by the Hensel's Lemma. This implies
that there is $(x_p,y_p)_{p\leq \infty}\in \prod_{p\leq \infty}
\bold X(\Bbb Z_p)$ such that
$$\psi_{\Theta/E}(\tilde f_E[(x_p,y_p)_{p\leq
\infty}])=1 . $$

When $D_1=\emptyset$, one needs further to compute
$\psi_{\Theta/E}(\tilde f_E[(x_2,y_2)])$. Since there is $\delta\in
\Bbb Z_2^\times $ with
$$ \delta \equiv \begin{cases} \pm 1 \mod 8  \ \ \ & \text{if
$n_1\equiv \pm 1 \mod 16$} \cr \pm 3 \mod 8 \ \ \ & \text{if
$n_1\equiv \pm 9 \mod 16$} \end{cases} $$ such that
$$\delta^2 = \begin{cases} 17^{s_2}n_1 \ \ \ & \text{if $n_1\equiv
1, \ 9 \mod 16$} \cr -17^{s_2} n_1 \ \ \ & \text{if $n_1 \equiv -1,
\ -9 \mod 16$ } \end{cases} $$ by the Hensel's lemma, one has
$$N_{E_v/\Bbb Q_2}(\beta^{s_1}\delta \prod_{p_i\in D_2}p_i^{\frac{1}{2}e_i})=\begin{cases} n \ \ \ &
\text{if $n_1 \equiv 1, \ 9 \mod 16$} \cr -n \ \ \ & \text{if $n_1
\equiv -1, \ -9 \mod 16$ } \end{cases} $$ where $v$ is the unique
prime of $E$ above $2$ and $\beta=6-\sqrt{34}$. For any
$(x_2,y_2)\in \bold X(\Bbb Z_2)$, one has
$$\psi_{\Theta/E}(\tilde f_E[(x_2,y_2)])=\begin{cases}
(\beta^{s_1}\delta \prod_{p_i\in D_2}p_i^{\frac{1}{2}e_i},\beta)_v \
\ \ & \text{if $n_1\equiv 1, \ 9 \mod 16$} \cr -(\beta^{s_1}\delta
\prod_{p_i\in D_2}p_i^{\frac{1}{2}e_i},\beta)_v \ \ \ & \text{if
$n_1 \equiv -1, \ -9 \mod 16$ }
\end{cases} $$
by the computation in Lemma \ref{computation}. Since
$$ \aligned & (\beta^{s_1}\delta \prod_{p_i\in D_2}p_i^{\frac{1}{2}e_i},\beta)_v
=(\beta,\beta)_v^{s_1}(\delta \prod_{p_i\in
D_2}p_i^{\frac{1}{2}e_i}, \beta)_v  \cr = &
(\beta,-1)_v^{s_1}(\delta \prod_{p_i\in D_2}p_i^{\frac{1}{2}e_i},
2)_2=(2,-1)_2^{s_1}(\delta \prod_{p_i\in
D_2}p_i^{\frac{1}{2}e_i},2)_2 \cr = & (\delta,2)_2 \prod_{p_i\in
D_2}(\frac{2}{p_i})^{\frac{1}{2}e_i}  =
\begin{cases} \prod_{p_i\in D_2}(\frac{2}{p_i})^{\frac{1}{2}e_i} \ \ \ & \text{if
$n_1\equiv \pm 1 \mod 16$} \cr -\prod_{p_i\in
D_2}(\frac{2}{p_i})^{\frac{1}{2}e_i} \ \ \ & \text{if $n_1\equiv \pm
9 \mod 16$,} \end{cases} \endaligned $$ one concludes that
$$\psi_{\Theta/E}(\tilde f_E[(x_2,y_2)])=\begin{cases} \prod_{p_i\in D_2}(\frac{2}{p_i})^{\frac{1}{2}e_i} \ \ \ & \text{if
$n_1\equiv 1, \ -9 \mod 16$} \cr -\prod_{p_i\in
D_2}(\frac{2}{p_i})^{\frac{1}{2}e_i} \ \ \ & \text{if $n_1\equiv -1,
\ 9 \mod 16$ }\end{cases} $$ for any $(x_2,y_2)\in \bold X(\Bbb
Z_2)$. By Proposition \ref{extra}, $\bold X(\Bbb Z)\neq \emptyset$
if and only if
$$\prod_{p_i\in D_4}(-1)^{e_i}
\prod_{p_i\in D_2}(\frac{2}{p_i})^{\frac{1}{2}e_i}
=\begin{cases} (-1)^{s_2} \prod_{p_i\in D_2}(\frac{2}{p_i})^{\frac{1}{2}e_i} &\text{if } n_1\equiv 1, -9\mod 16 \\
 (-1)^{s_2+1}\prod_{p_i\in D_2}(\frac{2}{p_i})^{\frac{1}{2}e_i} & \text{if } n_1\equiv -1, 9\mod 16 . \end{cases}$$
The proof is complete.
\end{proof}

\section{Examples for Higher Dimensional Tori}

In this section, we further explain how effective the method in \S
\ref{sec.adm} is for high dimension tori by providing the explicit
examples. Fix an integral solution $(x_0,y_0,z_0)$ of the equation
$x^2-2ly^2=2z^2$ for $l\equiv 1 \mod 8$ such that $x_0>0$ and
$(x_0,y_0,z_0)=1$. Let
$$L=\Q(\sqrt{2l}) \ \ \ \text{and} \ \ \ \Theta=L(\sqrt{x_0-y_0\sqrt{2l}}). $$ Then $\Theta$
is totally real and $\Theta/L$ is unramified over all primes except
the prime above $2$ and $2$ is totally ramified in $\Theta/\Bbb Q$.

\begin{prop}\label{higher} Let $l$ be a prime with $l\equiv 1 \mod 8$
satisfying $2l=r^2+s^2$ for two integers $r$ and $s$ with $s \equiv
\pm 3 \mod 8$ and $\Theta$ and $L$ be as above. Suppose $E$ is a
field containing $L$ and $\bold X$ is the scheme over $\Z$ defined
by the equation
$$f(x_1,\dots,x_k)=N_{E/\Q}(x_1e_1+\cdots+x_se_k)= n$$ where $\{e_1,\dots,e_k\}$ is a basis of
$\frak o_E$ over $\Z$ and $n$ is an integer. Then $\bold X(\Bbb
Z)\neq \emptyset$ if and only if $X_\Q(\Q)\neq \emptyset$ and there
is $$({x_1}_p,\dots, {x_k}_p)_{p\leq \infty}\in \prod_{p\leq \infty}
\bold X(\Bbb Z_p)$$ such that both
$$\psi_{H/E}[({x_1}_p e_1+\cdots+{x_k}_pe_k)_{p\leq \infty}] \ \ \ \text{and} \
\ \ \psi_{\Theta/L}[(N_{E/L}({x_1}_p e_1+\cdots+{x_k}_pe_k))_{p\leq
\infty}]
$$ are trivial, where $H$ is the Hilbert class field of $E$ and both
$\psi_{H/E}$ and $\psi_{\Theta/L}$ are the Artin maps.
\end{prop}
\begin{proof}
Since $G=R^1_{E/\Q}(\G_m)$, one has
$$Stab_\A({\bf X})= \{ (x_p)\in \prod_{p\leq \infty} \frak o_{E_p}^\times : \ N_{E/\Q}((x_p))=1 \in \Bbb I_\Q \}.$$
By the computation in Lemma \ref{computation}, one has $(\xi,
x_0-y_0\sqrt{2l})_v=1$ for any $\xi\in L_v^\times $ with
$N_{L_v/\Bbb Q_2}(\xi)=1$ where $v$ is the unique prime of $L$ above
$2$. This implies that
$$N_{E/L}(Stab({\bf X}(\Z_2))\subseteq N_{\Theta_\frak V/L_v}(\Theta_\frak
V^\times) $$ where $\frak V$ is the unique prime of $\Theta$ above
$v$. Since $\Theta/L$ is unramified over all primes except $v$, one
concludes that $$N_{E/L}(Stab_\A({\bf X}))\subseteq
N_{\Theta/L}(\Bbb I_{\Theta})$$ where $\Bbb I_{\Theta}$ is the idele
group of $\Theta$. Then the group homomorphism
$$\aligned f_E : \ G(\A_\Q)/G(\Q)Stab_\A({\bf X}) & \longrightarrow [\Bbb I_L/(L^\times  N_{\Theta/L}(\Bbb
I_{\Theta}))] \times [\Bbb I_E/(E^\times \prod_{p\leq \infty} \frak
o_{E_p}^\times)] \\(x_p)_{p\leq \infty} & \mapsto (N_{E/L}(x_p),
(x_p))_{p\leq \infty} \endaligned $$ is well-defined. By Lemma
\ref{computation} and the same argument in Proposition \ref{extra},
one obtains that $f_E$ is injective and the result follows.
\end{proof}

For any integer $n$, one can write $$n=(-1)^{s_0}
2^{s_1}17^{s_2}{p_1}^{e_1}\cdots {p_g}^{e_g} \ \ \ \text{and} \ \ \
D(n)=\{p_1, \cdots, p_g \} .$$ Decompose $D(n)$ into the disjoint
union of the following subsets
$$\aligned & D_1=\{p\in D(n) : \ (\frac{2}{p})=(\frac{17}{p})=-1  \}  \ \ \text{and} \ \ D_2=\{p\in D(n): \
(\frac{34}{p})=-1 \} \cr
 & D_3= \{p\in D(n): \
(\frac{2}{p})=(\frac{17}{p})=1 \ \text{and} \ x^4-12x^2+2 \equiv 0 \
mod \ p \ \ \text{is solvable} \} \cr &  D_4= \{p\in D(n): \
(\frac{2}{p})=(\frac{17}{p})=1 \ \text{and}\ x^4-12x^2+2 \equiv 0 \
mod \ p \ \ \text{is not solvable} \} . \endaligned $$ Let
$$n_1=(-1)^{s_0} \prod_{p_i\in D(n)\setminus D_2} p_i^{e_i} .$$
One can have one more explicit example.

\begin{exa} Let $n$ be an integer with the above notation and
$E=\Q(\sqrt{5},\sqrt{34})$. Then $n\in N_{E/\Q}(\frak o_E)$ if and
only if

(1)  $s_1\equiv s_2\equiv 0 \mod 2$,  $n_1\equiv \pm 1 \mod 8$,
$(\frac{n_1}{17})=1$ and $(\frac{34}{p_i})=(\frac{5}{p_i})=1$ for
odd $e_i$

(2)  $D_1\neq \emptyset$; or $$\sum_{p_i\in D_4} e_i \equiv
\begin{cases} 0 \mod 2 \ \ \ & \text{if $n_1\equiv 1, -9 \mod 16$} \cr
 1 \mod 2 \ \ \ & \text{if $n_1\equiv -1, 9 \mod 16$} \end{cases} $$ for $D_1=\emptyset$.
\end{exa}
\begin{proof} We apply Proposition \ref{higher} by taking
$E=\Q(\sqrt{5},\sqrt{34})$, $L=\Q(\sqrt{34})$ and
$\Theta=L(\sqrt{6-\sqrt{34}})$. The Hilbert class field of $E$ is
$E(\sqrt{2})$.

The local condition $n\in N_{E/\Q}(\frak o_{E_p})$ for all primes
$p\leq \infty$ is equivalent to the condition (1). Since
$Gal(E/\Q)=Gal(E_v/\Q_2)$, one concludes that $n\in
N_{E/\Q}(E^\times)$ by Theorem 6.11 in \cite{PR94}.

If $(x_p)_{p\leq \infty}$ is the local solution with $x_p\in \frak
o_{E_p}$ and $N_{E/\Q}(x_p)=n$ for all $p\leq \infty$, one can
verify that $\psi_{H/E}((x_p)_{p\leq \infty})=1$. The computation of
$\psi_{\Theta/L}[(N_{E/L}(x_p))_{p\leq \infty}]$ is the same as that
in Example \ref{pell}.
\end{proof}

\bigskip

\bf{Acknowledgment} \it{ We would like to thank Jean-Louis
Colliot-Th\'el\`ene for introducing us to study the Brauer-Manin
obstruction and a lot of helpful discussion. He also pointed out a
flaw in the original version of the paper. We would also like to
thank David Harari for sending us his paper \cite{Ha08}. Both
authors would also like to thank the referee for the suggestion of
improving the paper. The work is supported by the Morningside Center
of Mathematics. The first author is supported by NSFC, grant \#
10671104 and grant \# 10901150. The second author is supported by
NSFC, grant \# 10325105 and \# 10531060. }


\end{document}